\documentclass[a4paper,11pt]{article}

\usepackage{epsfig,amsmath,amssymb,amsthm}
\usepackage{openbib}
\usepackage[longnamesfirst,round]{natbib}
\usepackage[ansinew]{inputenc}
\usepackage{graphicx}
\usepackage{enumerate}
\usepackage{fullpage}
\usepackage{setspace}
\usepackage{dsfont}
\usepackage{mathrsfs}
\usepackage{verbatim,color}

\usepackage{xr}
\relax

\DeclareMathOperator{\cov}{Cov}

\newtheorem{lemma}{Lemma}[section]
\newtheorem{theorem}[lemma]{Theorem}

\newtheorem{definition}[lemma]{Definition}
\newtheorem{corollary}[lemma]{Corollary}

\newtheorem{model}[lemma]{Model}
\newtheorem{assumption}[lemma]{Assumption}

\newcommand{\calA}{\mathcal{A}}
\newcommand{\calB}{\mathcal{B}}
\newcommand{\Ee}{\mathscr{E}} 
\newcommand{\F}{\mathscr{F}}

\newcommand{\N}{\mathds{N}}  
   
\newcommand{\R}{\mathds{R}}   
\newcommand{\Z}{{\mathds{Z}}} 

\newcommand{\bx}{\boldsymbol{x}}
\newcommand{\by}{\boldsymbol{y}}
\newcommand{\bz}{\boldsymbol{z}}
\newcommand{\bK}{\boldsymbol{K}}
\newcommand{\bU}{\boldsymbol{U}}
\newcommand{\bX}{\boldsymbol{X}}
\newcommand{\bY}{\boldsymbol{Y}}
\newcommand{\bZ}{\boldsymbol{Z}}

\newcommand{\bNull}{\mbox{\boldmath{$0$}}}
\newcommand{\bEins}{\mbox{\boldmath{$1$}}}

\newcommand{\Var}{\mathrm{Var}}

\newcommand{\cid}{\stackrel{\mbox{\tiny d}}{\longrightarrow}} 
\newcommand{\cip}{\stackrel{\mbox{\tiny $p$}}{\longrightarrow}} 
\newcommand{\claw}{{\stackrel{\mbox{\tiny d}}{\longrightarrow}}}

\newcommand{\be}{\begin{equation}}
\newcommand{\ee}{\end{equation}}

\newcommand{\ben}{\begin{enumerate}[a)]} 
\newcommand{\bEN}{\begin{enumerate}[1.]}
\newcommand{\een}{\end{enumerate}}
\newcommand{\eEN}{\end{enumerate}}
\newcommand{\bit}{\begin{itemize}}
\newcommand{\eit}{\end{itemize}}

\DeclareMathOperator*{\argmax}{arg\,max}
\DeclareMathOperator*{\esssup}{ess\,sup}

\makeatletter
\def\@seccntformat#1{\csname the#1\endcsname. }
\def\section{\@startsection {section}{1}{\z@}{-3.5ex plus -1ex minus
    -.2ex}{1.3ex plus .2ex}{\center\large\sc}}
\def\subsection{\@startsection{subsection}{2}{\z@}{3.25ex plus 1ex minus .2ex}{-1em}{\normalsize\bf}}

\def\subsubsection{\@startsection{subsubsection}{3}{\z@}{3.25ex plus 1ex minus .2ex}{-1em}{\normalsize\it}}

\title{\textsc{Testing for changes in Kendall's tau}\footnote{The authors wish to thank their colleague Roland Fried for several very stimulating discussions that motivated this paper. Moreover, we are grateful for helpful comments from the editors and referees, which substantially improved a previous version of the paper. We are also indebted to Alexander D\"urre, who did a thorough proofreading of the manuscript. The research was supported in part by the 
Collaborative Research Grant 823 {\em Statistical modelling of nonlinear dynamic processes} of the German Research 
Foundation.}}

\author{
	\textsc{Herold Dehling} \\
	\textit{\small University of Bochum} \\[3.0ex]
	\textsc{Daniel Vogel}\footnote{Corresponding author: Daniel Vogel, E-Mail: daniel.vogel@abdn.ac.uk.} \\
	 \textit{\small University of Aberdeen} \\[3.0ex]
	\textsc{Martin Wendler} \\ 
	 \textit{\small University of Greifswald}\\[3.0ex]
		\textsc{Dominik Wied} \\
 \textit{\small  University of Cologne}
}

\date{}

\begin{document}

\maketitle
\newpage
{\bf Running Head:} Testing for changes in Kendall's tau

\bigskip
{\bf Proofs should be sent to:} Daniel Vogel, daniel.vogel@abdn.ac.uk


\bigskip
\begin{center}
{\bf Abstract}
\end{center}
\noindent
For a bivariate time series $((X_i,Y_i))_{i=1,\ldots,n}$ we want to detect whether the
correlation between $X_i$ and $Y_i$ stays constant for all $i = 1,\ldots,n$.
We propose a nonparametric change-point test statistic based on Kendall's tau. The asymptotic distribution under the null hypothesis of no change follows from a new $U$-statistic invariance principle for dependent processes. Assuming a single change-point, we show that the location of the change-point is consistently estimated.
Kendall's tau possesses a high efficiency at the normal distribution,
as compared to the normal maximum likelihood estimator, Pearson's moment correlation. Contrary to Pearson's
correlation coefficient, it shows no loss in efficiency
at heavy-tailed distributions, and is therefore particularly suited for financial data, where heavy tails are common.
We assume the data $((X_i,Y_i))_{i=1,...,n}$ to be stationary and $P$-near epoch dependent on
an absolutely regular process. The $P$-near epoch dependence condition constitutes a generalization of the usually considered $L_p$-near epoch dependence allowing for arbitrarily heavy-tailed data. We investigate the test numerically, compare it to previous proposals, and 
illustrate its application with two real-life data examples.






\newpage


 \section{Introduction} 

For risk management and portfolio optimization, the dependence between financial asset prices is of enormous importance. It is often assumed to be constant over the observed time period, which is a simplifying assumption that is evidently violated for longer observation periods. For good statistical modeling and successful decision making it is essential to detect changes in the association of financial price processes and, within reasonable time frames, re-estimate the correlation parameters. Particularly, in times of global financial crises, the price processes of most financial assets tend to be highly dependent, united in their common downward trend, causing the hedging powers of investment diversification to cease --- an effect for which the term \emph{diversification meltdown} has been coined. 

The problem of detecting changes in the distribution of sequential observations has a long history in statistics, see e.g.~\citet{csorgo:horvath:1997}. However, particularly detecting changes in the dependence structure of multivariate time series has attracted the focus of statistical research only recently. Examples for such detection procedures are \citet{loretan:phillips:1994}, who test for covariance stationarity of a possibly heavy-tailed time series, \citet{giacomini:hardle:spokoiny:2009}, who consider tests for homogeneity of time-varying copulae, \citet{aue:2009}, who propose a test for a constant covariance matrix, and \citet{wied:kraemer:dehling:2011}, who suggest a change-point test for correlations between two random variables based on Pearson's correlation coefficient.

With this paper, we want to contribute to the literature by proposing a new test for constant Kendall's tau that can be applied to dependent series. We recommend to use the rank correlation measure Kendall's tau instead of Pearson's correlation coefficient because it is almost as efficient as the moment correlation at normality, but is significantly more efficient at heavy-tailed distributions. For details see Section \ref{sec:props}. This issue is very important in finance and economics, where many key variables, including financial returns and foreign exchange rates, are commonly known to be heavy-tailed. 

\citet{gombay:horvath:1999}, \citet{quessy:said:favre:2013} study tests for changes in the dependence of multivariate time series based on Kendall's tau, but only consider independent observations. Moreover, these authors also use a bootstrap approximation for deriving critical values of a test statistic, while we provide a consistent long-run variance estimator, and do not need to rely on the bootstrap. A recent reference is \citet{buecher:kojadinovic:2016}, who propose change-point tests for Kendall's tau under mixing conditions, but not under the concept of $P$-near epoch dependence (see below).
Our change-point test does not require information on the position of potential break points. This is a structural similarity with many other tests in the econometrics and statistics literature, \citep[e.g.][]{inoue:2001}. See also Section 5.1 in \citet{stock:1994} and the references therein. 
%

Kendall's tau is a $U$-statistic. 
The main tool in analyzing the asymptotic behavior of the test statistic is a new functional limit theorem for sequential $U$-statistics processes for dependent data. This theorem is of interest in its own right. Allowing unbounded kernels and placing no moment requirement on the data process, it is formulated in by far greater generality than necessary for Kendall's tau, where the $U$-kernel is bounded and, further, the whole analysis can be restricted to bounded data sequences by an invariance argument. This functional limit theorem provides the basis for constructing change-point tests in the same spirit for any quantity that may be expressed as a $U$-statistic.

Several authors have used $U$-statistics for change-point problems before \citep[e.g.][]{gombay:horvath:1995,gombay:horvath:1999}.
The main contribution of the present paper is the thorough treatment of dependent series. 
We consider approximating functionals of mixing processes, where the approximation is in probability, not in an $L_p$ sense as in the usual near epoch dependence condition. We call this approximation concept $P$-near epoch dependence. It generalizes $L_p$-near epoch dependence, not requiring the existence of any moments, and is hence a fitting framework for nonparametric and robust data analysis.

Another appealing property of our approach in terms of broad applicability is the lack of assumptions on the copula between the two random variables $X_i$ and $Y_i$ for a fixed $i$. Although Kendall's tau is a dependence measure that only depends on the copula, there are no conditions on the existence of partial derivatives. In particular, distributions with non-zero tail dependence are included in our assumptions, which is important in empirical finance \citep[e.g.][]{patton:2006}. See \citet{segers:2012} for a discussion on this issue. 

The paper is organized as follows: Section \ref{sec:inf} contains the main theoretical results about $U$-statistics, that is, the functional limit theorem for $U$-statistics and the consistency result for the estimator of the long-run variance.  In Section \ref{sec:main}, the asymptotic properties of the test statistic under the null hypothesis are given. Section \ref{sec:identification} deals with estimating the location of a potential change-point. 
In Section \ref{sec:props}, the test is compared to previous proposals. Section \ref{sec:sim.res} contains a numerical study, where we observe
%
%
that the efficiency properties of Kendall's tau and Pearson's correlation coefficient translate into similar size and power properties of the corresponding change-point tests.  Section \ref{sec:example} demonstrates applications to financial data examples. 
Appendix \ref{sec:app:pned} contains further background on the concept of $P$-near epoch dependence, Appendix \ref{sec:app:local} investigates the behavior of the $U$-statistic process under a sequence of local alternatives, and Appendix \ref{sec:app:proofs} contains the proofs for the theorems of the main text. 
The proofs of the lemmas in the appendix and further technical results can be found in the online supplement. Readers may refer to the supplementary material associated with this article, available at Cambridge Journals Online (journals.cambridge.org/ect).

We use bold type face to denote vector-valued objects. Throughout, $|\cdot|_p$ denotes the $p$-norm in $\R^d$, $p \in [1,\infty)$, $d \in \N$. To denote the $L_p$ norm $\left( E |X|^p \right)^{1/p}$ of a real-valued random variable $X$, we occasionally write $||X||_p$, $p \in [1,\infty)$. All random variables are defined on a common probability space $(\Omega,\F,P)$.


 \section{Invariance Principle for $U$-Statistics of dependent series}
\label{sec:inf}

We treat Kendall's tau in the framework of asymptotic $U$-statistic theory. We first give a functional central limit theorem for general $U$-statistics for multivariate and short-range dependent time series. We further devise an estimator for the long run variance term and show its consistency.
Throughout this section, let $(\bX_i)_{i\in\Z}$ be a strictly stationary sequence of $d$-dimensional random variables with ($d$-dimensional) distribution function $F$. Let further $g:\R^d\times\R^d\rightarrow\R$ be a measurable, symmetric function. We call 
\begin{equation*}
 U_n = U_n\left(g\right) =\frac{2}{n(n-1)}\sum_{1\leq i<j\leq n} g\left(\bX_i,\bX_j\right)
\end{equation*}
the $U$-statistic with kernel $g$. The essential tool to treat $U$-statistics asymptotically is the Hoeffding decomposition into a linear and degenerate part, i.e., 
\begin{equation*}
	U_n\left(g\right) = 
	U + 
	\frac{2}{n}\sum_{i=1}^{n}g_{1}\left(\bX_{i}\right)
	+\frac{2}{n\left(n-1\right)}\sum_{1\leq i<j\leq n}g_{2}\left(\bX_{i},\bX_{j}\right),
\end{equation*}
where
\be \label{eq:hoeffding}
U  = E g(\bX,\bY), \qquad
g_1(\bx) = Eg(\bx,\bY)-U, \qquad
g_2(\bx,\by) =g(\bx,\by) - g_1(\bx) -g_1(\by) -U,
\ee
and $\bX$, $\bY$ are independent copies of $\bX_0$.

Concerning the serial dependence structure of the process $(\bX_i)_{i\in\Z}$, we assume it to be near epoch dependent in probability ($P$-NED) on an absolutely regular process. For the formal statement of this short-range dependence assumption, which follows below, it is convenient to let the process $(\bX_i)$ be indexed by $\Z$. The observed data is then the positive branch of the doubly infinite sequence.
For two sub-$\sigma$-fields $\calA$, $\calB$ of $\F$, we define the absolute regularity coefficient
\[     
	 	\beta(\calA,\calB) = E \left[ \esssup\left\{ | P(A|\calB) - P(A)| \, : \, A \in \calA  \right\} \right].
\]
The absolute regularity coefficient is a measure of dependence between the $\sigma$-fields $\calA$ and $\calB$, it lies between $0$ and $1$, and equals 0 if $\calA$ and $\calB$ are independent.
\begin{definition}\rm  \label{def:dep}
Let $(\bX_n)_{n \in \Z}$ and $(\bZ_n)_{n \in \Z}$ be $d$- and $r$-variate stochastic processes on $(\Omega,\F,P)$, respectively, $d,r \ge 1$, such that the $(d + r)$-variate process $((\bX_n,\bZ_n))_{n \in \Z}$ is stationary. For $k \le n$, let $\F_k^n = \sigma(\bZ_k,\ldots,\bZ_n)$, where also $k = -\infty$ and $n = \infty$ are permitted.
\begin{enumerate}[(i)]
\item
The process $(\bZ_n)_{n \in \Z}$ is called \emph{absolutely regular} if the absolute regularity coefficients
\[
		\beta_k = \beta(\F^0_{\!\!-\infty},\F_k^{\infty}), \qquad k \ge 1, 
\]
converge to zero as $k \to \infty$. 
\item
The process $(\bX_n)_{n \in \Z}$ is called \emph{near epoch dependent in probability} or short \emph{$P$-near epoch dependent} ($P$-NED) on the process $(\bZ_n)_{n \in\Z}$ if there is a sequence of approximating constants $(a_k)_{k \in \N}$ with $a_{k} \to 0$ as $k \to \infty$, a sequence of functions $f_k:\R^{r \times (2k+1)} \to \R^d$, $k \in \N$, and a non-increasing function $\Phi: (0,\infty) \to (0,\infty)$ such that 
\be \label{eq:l_0}
	P \left( \left|  \bX_0 - f_k(\bZ_{-k},\ldots,\bZ_{k})  \right|_1 > \varepsilon\right) \ \le \ a_k \Phi(\varepsilon) 
\ee
for all $k \in \N$ and $\varepsilon > 0$. 
\end{enumerate}
\end{definition}
By Lemma \ref{lem:l_0} (iii) of Appendix A, $L_p$-near epoch dependence ($p \ge 1$) implies near epoch dependence in probability. So $P$-NED can be viewed as a generalization of the more frequently considered $L_2$-NED.  All the limit theorems in this article may be formulated for $L_p$-NED sequences as well. We prefer to use $P$-NED instead of $L_p$-NED, since we particularly want to include very heavy tailed data and do not want to assume the existence of even first moments. Further details and references on the different weak dependence concepts are given in Appendix \ref{sec:app:pned}.

For the functional $U$-statistic limit theorem we require Assumption \ref{ass3}, \ref{ass2} and \ref{ass1} to hold.
\begin{assumption}\label{ass3} The process $(\bX_i)_{i\in\Z}$ is $P$-NED on an absolutely regular sequence $(\bZ_i)_{i\in\Z}$, and there is a $\delta > 0$ such that 
\begin{equation*}
a_k\Phi(k^{-6}) = O (k^{-6(2+\delta)/\delta})
\qquad \mbox{ and } \qquad
\sum_{k=1}^\infty k\beta_k^{\delta/(2+\delta)} < \infty.
\end{equation*}
\end{assumption}

Furthermore, a moment condition on $g(\bX_i,\bX_j)$ is required. Note that we do not impose any moment conditions on the data sequence 
$(\bX_i)_{i\in\Z}$ itself. 
\begin{assumption}\label{ass2} There is a constant $M > 0$ such that for all $k,n\in\N$
\begin{equation*}
E\left|g\left(f_k(\bZ_{-k},\ldots,\bZ_k),f_k(\bZ_{n-k},\ldots,\bZ_{n+k})\right)\right|^{2+\delta}\leq M
\quad
\mbox{ and }
\quad
E\left|g(\bX_0,\bX_n)\right|^{2+\delta}\leq M.
\end{equation*}
\end{assumption}

Note that Assumptions \ref{ass3} and \ref{ass2} are linked via $\delta$. Weaker moment conditions have to be paid for by a faster decay of the short-range dependence coefficients and vice versa. The next assumption is also known as the variation condition and was introduced by \citet{denker:keller:1986}. It can be understood as a form of Lipschitz continuity of the kernel $g$ with respect to $F$.

\begin{assumption}\label{ass1} There are constants $L, \epsilon_0 > 0$ such that for all $\epsilon\in(0,\epsilon_0)$
\begin{equation*}\label{line8}
 E\left(\sup_{|\bx-\bX|\leq \epsilon,|\by-\bY|\leq \epsilon}\left|g\left(\bx,\by\right)-g\left(\bX,\bY\right)\right|\right)^2\leq L\epsilon,
\end{equation*}
where $\bX$, $\bY$ are independent with the same distribution as $\bX_0$.
\end{assumption}
We are now ready to state the following weak invariance principle for the sequential $U$-process. The proof is given in Appendix C.
\begin{theorem}\label{theo1} Under Assumptions \ref{ass3}, \ref{ass2} and \ref{ass1}, we have
\begin{equation*}
			\left(\frac{[ns]}{\sqrt{n}}\left(U_{[ns]}(g)-U\right)\right)_{s\in[0,1]} \ \cid \ 2\sigma W
\end{equation*}
in $D[0,1]$, where $W$ denotes a standard Brownian motion, and the long run variance is given by
\begin{equation*}
		\sigma^2	=	 \sum_{r=-\infty}^\infty \cov\left(g_1(\bX_0),g_1(\bX_r)\right).
\end{equation*}
\end{theorem}

Without specific assumptions on the distribution of the whole process $(\bX_i)_{i\in\Z}$, the long run variance term $\sigma^2$ is unknown, and, even if one is willing to make such assumptions, it may yet be cumbersome to evaluate it. Thus, for statistical applications, an estimator of $\sigma^2$ is desired. For the sample mean, the problem of estimating the long run variance has already been studied by many authors. Our proposal for an estimate of $\sigma^2$ is based upon combining the HAC (heteroscedasticity and autocorrelation consistent) kernel estimator by \citet{dejong:2000} with an empirical version of the Hoeffding decomposition.

For the kernel $g$, we define the empirical version $\hat{g}_1$ of $g_1$ as 
\begin{equation*}
		\hat{g}_1(\bx) = \frac{1}{n}\sum_{i=1}^ng(\bx,\bX_i)-\frac{1}{n^2}\sum_{i,j=1}^ng(\bX_i,\bX_j),
\end{equation*}
and the empirical covariance for lag $r$ as $\hat{\rho}(r) =\frac{1}{n}\sum_{i=1}^{n-r}\hat{g}_1(\bX_i)\hat{g}_1(\bX_{i+r})$.
We then estimate $\sigma^2$ by
\be \label{eq:variance.est}
	\hat{\sigma}_n^2 = \hat{\rho}(0)+2\sum_{r=1}^{n-1} \kappa \left(\frac{r}{b_n}\right)\hat{\rho}(r),
\ee
where $\kappa$ is a weight function (or HAC kernel function) and $b_n$ a bandwidth depending on $n$. In order achieve consistency, $\kappa$ and $b_n$ have to fulfill some regularity conditions.
\begin{assumption}\label{ass4} The function $\kappa:[0,\infty)\rightarrow[-1,1]$ is continuous at 0 and at all but a finite number of points and $\kappa(0)=1$. Furthermore, $|\kappa|$ is dominated by a non-increasing,
integrable function and
\begin{equation*}
\int_0^\infty\left|\int_0^\infty \kappa(t)\cos(xt)dt\right|dx<\infty.
\end{equation*}
The bandwidth $b_n$ satisfies $b_n\rightarrow \infty$ as $n\rightarrow\infty$ and $b_n/\sqrt{n}\rightarrow0$.
\end{assumption}
Assumption \ref{ass4} mainly coincides with Assumption 1 of \citet{dejong:2000}. It is satisfied by a large class of kernels, in particular the popular Bartlett kernel $\kappa(t) = (1 - |t|) \bEins_{\{|t| \le 1\}}$.
The proof of the following consistency result is also given in Appendix C.

\begin{theorem}\label{theo3} Under Assumptions \ref{ass3}, \ref{ass2}, \ref{ass1} and \ref{ass4}, we have $\hat{\sigma}_n^2	\cip	\sigma^2$
as $n\rightarrow\infty$, where $\sigma^2$ is as in Theorem \ref{theo1}.
\end{theorem}

Theorems \ref{theo1} and \ref{theo3} give rise to a general test statistic 
\[
	 \hat{T}_n = \frac{1}{2 \hat\sigma_n} \max_{2 \le k \le n-1} \frac{k}{\sqrt{n}} | U_k - U_n |
\]
for CUSUM-type change-point tests based on $U$-statistics.
By combining the continuous mapping theorem (applied to the functional which maps $x \in D[0,1]$ to the real number $\sup_{0\le t \le 1} |x(t) - t x(1)|$) and Slutsky's lemma we arrive at the following result.
\begin{corollary} \label{cor:1}
Under Assumptions \ref{ass3}, \ref{ass2}, \ref{ass1} and \ref{ass4}, and if $\sigma^2 > 0$, we have $\hat{T}_n  \cid \sup_{0\leq \lambda \leq 1} |B(\lambda)|$, where $B$ is a standard Brownian bridge on $[0,1]$.
\end{corollary}

The assumptions of Theorems \ref{theo1} and \ref{theo3} are very broad and easy to verify. The kernel $g$ as well as the marginal distribution $F$ are, except for the variation condition, completely arbitrary. Furthermore, all time series models relevant in financial applications fulfill our short-range dependence condition with exponential decay of $a_k$ and $\beta_k$. For further details on how $P$ NED is related to the usual $L_2$ near epoch dependence, see Appendix \ref{sec:app:pned}.

Invariance principles for $U$-statistics similar to Theorem \ref{theo1} were established by \citet{yoshihara:1976} for absolutely regular processes, which do not cover many time series models. Central limit theorems for $U$-statistics have been investigated under more general conditions: \citet{denker:keller:1986} considered Lipschitz continuous functionals of absolutely regular sequences and
\citet{borovkova:burton:dehling:2001} $L_1$-NED processes. 
As far as we know, functional central limit theorems (invariance principles) for $U$-statistics have not been studied under more general dependence conditions.

The potential applications of Corollary \ref{cor:1} are manifold. Several $U$-statistics have gained popularity as estimators that combine high efficiency under normality with appealing invariance and robustness properties (in the classical sense of robust statistics). The leading example is certainly Kendall's tau, which we will study in depth in Sections \ref{sec:main} through \ref{sec:example}. Another prominent example is Gini's mean difference 
\[
	g_n = \frac{2}{n(n-1)} \sum_{1 \le i <j \le n} |x_i - x_j|
\]
for univariate data $x_1, \ldots, x_n$. Similarly to Kendall's tau, Gini's mean difference is, as a measure of scale, almost as efficient as the maximum likelihood estimator at normality (the standard deviation), but is more efficient than the latter at heavy-tailed distributions and less sensitive to single outlying observations \citep{gerstenberger:vogel:2015}. Thus Corollary \ref{cor:1} directly yields the asymptotic null distribution of a Gini's mean difference based change-point test for scale, which requires only $2+\delta$ moments as compared to $4+\delta$ moments for a sample-variance-based test. Assumption \ref{ass1} is automatically fulfilled in this example as the corresponding kernel $g(x,y) = |x-y|$ is Lipschitz continuous.

Furthermore, the results can be straightforwardly extended to multivariate $U$-statistics. The process convergence of a $p$-dimensional $U$-statistic
\[
	\bU_n = (U_n^{(1)}, \ldots, U_n^{(p)}),
\]
where $p$ is generally different from the data dimension $d$, is obtained by means of the Cram\'er--Wold device by considering the univariate $U$-statistic $\tilde{U}_n = \sum_i \lambda_i U_n^{(i)}$ for arbitrary $(\lambda_1, \ldots, \lambda_p) \in \R^p$. Similarly, a multivariate version of Theorem \ref{theo3} follows from an entry-wise consideration as convergence in probability of a random matrix is implied by the convergence of its marginals. 
Besides vector-valued versions of Kendall's tau or Gini's mean difference to test for changes in the rank correlation or scale, respectively, of several time series jointly, we can also consider the following estimator, 
\[
	\bK_n =  \frac{2}{n(n-1)} \sum_{1 \le i <j \le n} \frac{ (\bX_i-\bX_j)(\bX_i-\bX_j)^\top }{ |\bX_i-\bX_j|_2^2 }, 
\]
which is sometimes referred to as the \emph{spatial Kendall's tau matrix}. It is also known to possess a rather high efficiency at normality, which has led several authors to consider this estimator in various contexts \citep[e.g.][]{fan:liu:wang:2015}. This estimator allows to consistently estimate the eigenvectors and the ordering of the eigenvalues of the covariance matrix. Thus Theorems \ref{theo1} and \ref{theo3} also provide the theoretical foundation for a robust change-point test for, say, detecting changes in the leading eigenvector of the marginal covariance matrix of a multivariate time series. 

Further, these $U$-statistic results encompass all linear statistics, i.e., $U$-statistics of order one, such as the classical CUSUM test based on the sample mean. The classical change-point test for detecting changes in scale, as studied by \citet{inclan:tiao:1994}, is essentially the CUSUM test applied to the squares of the centered data. It requires some additional technical effort to thoroughly deal with the centering for the data, which is ignored by some authors and dealt with in different ways by other authors. Employing the $U$-statistic representation of the sample variance with kernel $g(x,y) = (x-y)^2/2$, our results provide another elegant method of proof.

A referee raised the question whether the tail dependence coefficient could also be treated in the current $U$-statistic framework. In this situation, one would consider a $U$-statistic of order one with kernel
\begin{equation}\label{TDC}
		g(F_X(x_i),F_Y(y_i)) = \sqrt{\frac{n}{k}} \bEins_{\left\{ F_X(x_i) \leq k/n,\, F_Y(y_i) \leq k/n\right\}}
\end{equation}
for a bivariate time series $((X_i,Y_i))_{i\in\Z}$, where $F_X$ and $F_Y$ denote the marginal distributions of $X_0$ and $Y_0$, respectively, and $k = k(n)$ converges to $\infty$ and $k/n$ converges to $0$. Using this kernel in Theorem \ref{theo1}, one obtains the process $\mathbb{B}_n$ from \citet{buecher:jaeschke:wied:2014}, which provides the basis for a CUSUM-type test statistic for detecting changes in tail dependence. Thus one would need to extend the current setting to a situation where the kernel $g$ may depend on $n$. More crucially, one would need a theorem in which boundedness of the second moments of the kernel is sufficient. While $g^2$ converges to the tail dependence coefficient, $g^{2+\delta}$ diverges unless the tail dependence coefficient is zero. This seems difficult to achieve in the current serial dependence setting. A solution could be to restrict the dependency conditions, e.g.\ to uniformly mixing processes.
An alternative way of treating tail dependence is by employing the direct link between the tail dependence coefficient and Kendall's tau in specific parametric models. By testing for constant Kendall's tau and assuming a specific model is true, one automatically tests for constant tail dependence as well. However, such an approach may be considered not optimal since, in contrast to an approach based on \eqref{TDC}, it strongly relies on the model assumption, see the discussion in \citet{buecher:jaeschke:wied:2014}.


 \section{Change-point Detection for Kendall's tau}
\label{sec:main}

Let $((X_i,Y_i))_{i\in\Z}$ be a strictly stationary series of bivariate random vectors with marginal distribution function $F(x,y) = P(X_0\leq x, Y_0\leq y)$. Throughout the remainder of the article, we assume $F$ to be Lipschitz continuous. This is fulfilled if $F$ possesses a bounded density, but, e.g., $X_0 = Y_0$ almost surely is also allowed. 
Kendall's rank correlation coefficient, also known as Kendall's tau, is defined as
\[
 	\tau = P((X^\prime-X)(Y^\prime-Y)>0) - P((X^\prime-X)(Y^\prime-Y) < 0),
\]
where $(X,Y)$ and $(X^\prime,Y^\prime)$ are two independent random variables with distribution function $F$. 
Kendall's tau is a measure of correlation between the random variables $X$ and $Y$, where we understand correlation generally as \emph{monotone dependence}, which, loosely speaking, carries information on to what extent one variable on average increases or decreases as the value of the other increases. Kendall's tau, as well as the related dependence measure Spearman's rho, is a function of the copula only. In particular, it does not depend on the marginal distributions and is hence invariant to monotone marginal transformations, see 
\citet[e.g.][Chap.~5]{nelsen:2006}. Consequently, its sample version only depends on the ranks of the data, which is
the reason for Kendall's tau being also referred to as \emph{rank correlation measure}.
The sample version of Kendall's tau is defined as
\[
 \hat{\tau}_n = \frac{2}{n(n-1)} \sum_{1 \le i <j \le n} 
 								\left[ 
 											\bEins_{(0,\infty)}\left\{ (X_j-X_i)(Y_j-Y_i) \right\} 
 											- \bEins_{(-\infty,0)}\left\{(X_j-X_i)(Y_j-Y_i) \right\}  
 								\right],
\]
which is a $U$-statistic with kernel function $g:\R^2\times\R^2\rightarrow \R$ given by
\begin{equation} \label{eq:kernel}
	g\left((x,y),(x^\prime,y^\prime)\right ) = 
	\mathbf{1}_{(0,\infty)} \{(x^\prime-x)(y^\prime-y)\} - \mathbf{1}_{(-\infty,0)}\{(x^\prime-x)(y^\prime-y)\}.
\end{equation}
For dependent data, $\hat\tau_n$ is not necessarily unbiased, but under the weak dependence conditions we consider, it is (strongly) consistent for $\tau$.
We will study a test for change in the dependence structure of the marginals by the test statistic
\[
	 \hat{T}_{\tau,n} = \max_{k = 1, \ldots, n}\frac{k}{\sqrt{n}}|\hat{\tau}_k-\hat{\tau}_n|,
\]
which rejects the null hypothesis of constant rank correlation if $\hat{T}_{\tau,n}$ is too large. 
By means of the functional limit theorem for sequential $U$-statistics (Theorem \ref{theo1}), we know that
\begin{equation*}
		\left(\frac{[tn]}{\sqrt{n}}(\hat{\tau}_{[tn]}-\tau)\right)_{t\in[0,1]}
\end{equation*}
converges under the assumption of no change and for weakly dependent series weakly to a Brownian motion $2 \sigma_\tau W$, where 
\begin{equation} \label{eq:lrv.tau}
	\sigma^2_\tau =  \sum_{j=-\infty}^\infty  E \left[\psi(X_0,Y_0) \psi(X_j,Y_j)\right]
\end{equation}
with $\psi(x,y) = 4 F(x,y) - 2 F_X(x) - 2 F_Y(y) + 1 - \tau$. 
The function $\psi$ can be seen to be the linear part $g_1$ of the Hoeffding decomposition (\ref{eq:hoeffding}) for the specific kernel $g$ in (\ref{eq:kernel}) for any continuous distribution function $F$.
By the continuous mapping theorem, we have the weak convergence
\begin{equation*}
 \hat{T}_{\tau,n} = \max_{k = 1, \ldots, n}\frac{k}{\sqrt{n}}|\hat{\tau}_k-\hat{\tau}_n|
	\cid 2 \sigma_\tau \sup_{0\leq \lambda \leq 1} |B(\lambda)|,
\end{equation*}
where $B$ is a standard Brownian Bridge. The distribution of $\sup_{0\leq \lambda \leq 1} |B(\lambda)|$ is known and sometimes referred to as Kolmogorov distribution. We use the estimator for the long run variance proposed in Section \ref{sec:inf}. Let $F_n$, $F_{X,n}$ and $F_{Y,n}$ denote the empirical distribution functions of $((X_i,Y_i))_{i =1, \ldots, n}$, $(X_i)_{i =1,\ldots, n}$ and $(Y_i)_{i = 1,\ldots,n}$, respectively. 
Then 
$\hat{\psi}_{n,i} = 4 F_n(X_i,Y_i) - 2F_{X,n}(X_i)  - 2F_{Y,n}(Y_i) + 1 - \hat{\tau}_n$
can be seen to equal $\hat{g}_1((X_i,Y_i))$, and the variance estimator (\ref{eq:variance.est}) can be written as
\be \label{eq:D.hat}
	\hat{\sigma}^2_{\tau,n} \ = \ \frac{1}{n}\sum_{i=1}^n \hat{\psi}^2_{n,i} \ + \ 
	\frac{2}{n} \sum_{j =1}^{n-1} \kappa\left(\frac{j}{b_n}\right) \, 
				 \sum_{i=1}^{n-j} \hat{\psi}_{n,i} \hat{\psi}_{n,i+j},
\ee
where $\kappa$ and $b_n$ are the same as in Section \ref{sec:inf}.
Corollary \ref{th:change+variance-asy} below gives the asymptotic distribution of the test statistic $\hat{T}_n/ (2 \hat{\sigma}_{\tau,n})$ under the null hypothesis of no change. 
\begin{corollary} \label{th:change+variance-asy}
Let $((X_i,Y_i))_{i\in\Z}$ be a two-dimensional, stationary process with a Lipschitz continuous marginal distribution function. Assume that $((X_i,Y_i))_{i\in\Z}$ is $P$-NED with approximating constants $(a_k)_{k \ge 1}$ on an absolutely regular process with absolute regularity coefficients $(\beta_k)_{k \ge 1}$ satisfying Assumption \ref{ass3} for some $\delta > 0$. Let further Assumption \ref{ass4} hold. 
Then, if $\sigma^2_\tau > 0$,  
\begin{equation} \label{eq:change+variance-asy}
 \frac{\hat{T}_{\tau,n}}{2\hat{\sigma}_{\tau,n}}  \  \claw \ \sup_{0\leq \lambda \leq 1} |B(\lambda)|, 
\end{equation}
where $(B(\lambda))_{0\leq \lambda \leq 1}$ is, as before, a standard Brownian bridge.
\end{corollary}


 \section{Change-point Estimation and Local Power}
\label{sec:identification} 

If the test rejects the null hypothesis of constant correlation, and if it is furthermore reasonable to assume that there is one sudden change-point, it is of interest to locate this change-point. An intuitive estimator, which is common when dealing with CUSUM-type change-point tests, is
the position at which the weighted differences take their maximum, that is 
\[
	\hat{k}_n = \argmax_{1 \leq k \leq n} \frac{k}{\sqrt{n}} | \hat \tau_k - \hat \tau_n |.
\]
We will show in the following that this is indeed a reasonable estimator. We will assume that the following model holds.%
\begin{model}[Change-point model]
\rm
\label{mod:cp}
Let $0 < \lambda^* < 1$. For $n \in \N$ let $((X_{i}^{(n)},Y_{i}^{(n)}))_{1\le i \le [\lambda^* n]}$ and $((X_{i}^{(n)},Y_{i}^{(n)}))_{[\lambda^* n]+1\le i \le n}$ 
be two bivariate, stationary stochastic processes with marginal distribution functions $F$ and $G$, respectively. 
Let furthermore $((X_{i}^{(n)},Y_{i}^{(n)}))_{1\le i \le n}$ be $P$-near epoch dependent\footnote{
	For non-stationary processes the short-range dependence conditions have to be formulated slightly more generally than in 
	Definition \ref{def:dep}. The absolute regularity coefficients $(\beta_k)_{k \in \N}$ are defined as $\beta_k = 
	\sup_{t\in\Z}  \beta(\F_{-\infty}^t, \F_{t+k}^{\infty})$, and the $P$-NED approximation coefficients $(a_k)_{k \in 
	\N}$ must satisfy
	\[
		\sup_{t \in \Z} P\left(  \left|  \bX_t - f_{k,t}(\bZ_{t-k},\ldots,\bZ_{t+k})  \right|_1 > \varepsilon    \right) \ 			\le \ a_k \Phi_t(\varepsilon),	
	\]
	where the functions $f_{k,t}$ and $\Phi_t$ may also depend on $t$. The underlying process $(\bZ_t)_{t\in\Z}$ is not required to be stationary.
} 
on an absolutely regular process with coefficients satisfying Assumption \ref{ass3} uniformly for all $n$. 
\end{model}%
The goal is to estimate $\lambda^*$. Let $\tau_F$ and $\tau_G$ denote Kendall's tau of $F$ and $G$, respectively. Moreover, let $\tau_{F,G} = E g((X_1,Y_1), (X_2,Y_2))$, cf.\ (\ref{eq:kernel}), where $(X_1,Y_1) \sim F$ and $(X_2,Y_2) \sim G$ are independent.%
\begin{theorem} \label{th:identification}
If $((X_{i}^{(n)},Y_{i}^{(n)}))_{1\le i \le n, n \in\N}$ follows Model \ref{mod:cp}, furthermore $\tau_F \neq \tau_G$ and
\be \label{eq:strange.condition}
	\frac{(1-\lambda^*)^2 }{2 \left( (1-\lambda^*)^2 + \lambda^*\right) } \ \le \ \frac{\tau_{F,G}-\tau_F}{\tau_G-\tau_F} \ < \ 1,
\ee
then $\hat{k}_n/ n  \cip \lambda^*$  as $n \to \infty$.
\end{theorem}
Condition (\ref{eq:strange.condition}) prohibits $\tau_{F,G}$ to be too close to $\tau_{F}$ compared to $\tau_G$. It is an open research question which values of $\tau_{F,G}$ are possible for given $\tau_F$ and $\tau_G$, and in particular if $\tau_{F,G}$ may at all lie outside the interval $[\tau_F,\tau_G]$.
%
%
%
%

We further analyze the power of the proposed change-point test under local alternatives. We give a formula for the asymptotic distribution of the test statistic within a specific class of local alternatives.

\begin{model}
\label{mod:la}
{\rm
Let $(X_i,Y_i)_{i\in\Z}$ be a strictly stationary, bivariate process that satisfies Assumption \ref{ass3}, and let $0<\lambda^\ast <1$ and $\Delta>0$. We define for any integer $n\geq 1$ 
\[
 (X_i^{(n)},Y_i^{(n)})=\left\{  
 \begin{array}{ll}
 (X_i,Y_i) & 1\leq i\leq [n\lambda^\ast] \\
 (X_i+\frac{\Delta}{\sqrt{n}}Y_i,Y_i) & [n \lambda^\ast]\leq i\leq n.
  \end{array}
 \right.
\]
We assume that $(X_i,Y_i)$ has an absolutely continuous distribution $F$ with  bounded density $f$ satisfying $\lim_{x\to\infty} f(x,y) = 0$ and $\lim_{y\to\infty} f(x,y)=0$ for all $x,y \in \R$ and
\[
	\int \int \sup_{0\le |\alpha| \le\varepsilon} |y \frac{\partial}{\partial x} f(x + \alpha y, y)| d x d y < \infty
\]
for some $\varepsilon > 0$. 
By $F_n$ we denote the distribution of $(X_i^{(n)},Y_i^{(n)})$. 
The density of $F_n$ is given by
\[ 
 f_n(x,y)=f(x-\frac{\Delta}{\sqrt{n}}y,y).
\]
}
\end{model}

\begin{theorem} \label{th:localpower}
If $(X_i,Y_i)_{i\geq 1}$ follows Model \ref{mod:la}, we have
\[
 \hat{T}_{\tau,n} \cid \sup_{0\leq \lambda \leq 1} \left|2\sigma_\tau B(\lambda)+\Delta (\phi_{\lambda^\ast}(\lambda) -\lambda\phi_{\lambda^\ast}(1)) \right|,
\]
where $\sigma_\tau$ is defined in (\ref{eq:lrv.tau}) and the function $\phi_{\lambda^\ast}:[0,1]\rightarrow \R$ is defined as 
\[
  \phi_{\lambda^\ast}(s)=\left\{ 
  \begin{array}{ll}
   4 \left( \int_{-\infty}^\infty \big(\int_{-\infty}^\infty y (2F(y |x)-1) F(dy|x)    \big) f_X^2(x) dx\right) (s-\lambda^\ast)  & s\geq \lambda^\ast, \\
   0 & s\leq \lambda^\ast.
  \end{array}
  \right.
\]
\end{theorem}

From Theorem \ref{th:localpower} we can conclude the consistency of the test against local alternatives of the type studied in Model 
\ref{mod:la}, for which it suffices to observe that the integral occurring in the limit in Theorem \ref{th:localpower} is non-zero: For any absolutely continuous distribution function $F$ with density $f$, and with finite expectation, we have
\begin{eqnarray*}
\int_{-\infty}^\infty y(2\, F(y) -1) f(y) dy  &=& \int_{-\infty}^\infty 2 y F(y) f(y) dy - \int_{-\infty}^\infty y f(y) dy \\
 &=& \int y dF^2(y) -\int y dF(y)  \ = \ E\left( \max(Y_1,Y_2)\right) -E(Y_1),
\end{eqnarray*}
where $Y_1,Y_2$ are independent random variables with distribution $F$.  Now, $E\left( \max(Y_1,Y_2)\right) -E(Y_1)>0$,
unless $Y_1$ is a constant, in which case $E\left( \max(Y_1,Y_2)\right) -E(Y_1)=0$. 
Thus, the inner integral $\int_{-\infty}^\infty y(2F(y|x)-1)F(dy|x)$ is positive, unless the conditional distribution of $Y_1$ given $X_1=x$ is degenerate, i.e. $Y_1$ takes only one value. Hence the integral 
\[
  \int_{-\infty}^\infty \int_{-\infty}^\infty y(2F(y|x)-1)F(dy|x) f_X^2(x) dx
\]
is positive unless $Y_1$ is a deterministic function of $X_1$.


 \section{Previous Proposals}
\label{sec:props} 

\citet*{wied:kraemer:dehling:2011} consider the test statistic
\[
	\hat{T}_{\varrho,n} = \max_{1 \le k \le n} \frac{k}{\sqrt{n}} | \hat{\varrho}_k - \hat{\varrho}_n|
\]
based on Pearson's linear correlation coefficient $\hat{\varrho}_n$.
%
This test 
will serve as the main benchmark method for our test. Our motivation for using Kendall's tau is the wish to efficiently detect structural changes in arbitrarily heavy-tailed and potentially contaminated data. Both tests are constructed in a similar way. The differences between the two tests are largely due to the different properties of the estimators $\hat{\varrho}_k$ and $\hat{\tau}_k$. It is therefore worthwhile to have a brief look at these two correlation measures.

Kendall's tau is invariant with respect to strictly monotonic, componentwise transformations of $F$.
For continuous $F$, it can be written as $\tau = 4 E F(X,Y) - 1$, which can be seen to be a function of the copula of $F$ alone. The same applies to its asymptotic variance at i.i.d.\ data, $ASV(\hat{\tau}_n) = 16 \Var(\psi(X,Y))$ with $\psi$ being defined in Section \ref{sec:main}.  For further details see, e.g., \citet[][Chap.~5]{nelsen:2006}.
Consequently, no matter how heavy the tails of the distribution $F$ are, as long as the marginals are joined by a Gauss copula, the variance of $\hat{\tau}_n$ is the same as in the normal model.

In the normal model there is a one-to-one correspondence between $\tau$ and Pearson's moment correlation $\varrho$, which is given by
\be \label{eq:trafo}
	\tau = (2/\pi)\arcsin(\varrho), \qquad -1 \le \varrho \le 1.
\ee
Thus by letting $\hat{\varrho}_{\tau,n} = \sin(\pi (\hat{\tau}_n - 1/2))$,
the estimators $\hat{\varrho}_n$ and $\hat{\varrho}_{\tau,n}$ are, under normality, both Fisher-consistent for the same quantity $\varrho$. Comparing their asymptotic variances 
\be \label{eq:asv.normal}
	 ASV(\hat{\varrho}_n) = (1 - \varrho^2)^2,
\qquad	
ASV(\hat{\varrho}_{\tau,n}) =  (1-\varrho^2)(\pi^2/9 - 4\arcsin^2(\varrho/2))
\ee
\citep[e.g.][]{croux:dehon:2010},
allows a prognosis concerning the efficiency relation of the corresponding change-point tests. For example, for two independent random variables we have $ASV(\hat{\varrho}_n) = 1$ and $ASV(\hat{\varrho}_{\tau,n}) = \pi^2/9 = 1.097$.

The scope of the identity (\ref{eq:trafo}) extends to all elliptical distributions with finite second moments. \citep[e.g.][p.~97]{mcneil:frey:embrechts:2005}. 
For example, at the two-dimensional elliptical $t_{\nu}$-distribution with $\nu$ degrees of freedom and uncorrelated margins, we have
\[
	 ASV(\hat{\varrho}_n) = (\nu-2)/(\nu-4), \qquad \nu > 4,
\]\[
	 ASV(\hat{\varrho}_{t(\nu),n}) = (\nu+4)/(\nu+2), \qquad \nu > 0,
\]
where $\hat{\varrho}_{t(\nu),n}$ is the maximum likelihood estimator of $\varrho$ at the two-dimensional $t_{\nu}$ family. \citep[e.g.][p.~221]{bilodeau:brenner:1999}. 
\citet{dengler:2010} gives values for the asymptotic variance of $\hat{\varrho}_{\tau,n}$ at uncorrelated $t_{\nu}$-distributions. It is a decreasing function of $\nu$, it equals 1.922 and 1.296 for $\nu = 1$ and $\nu = 5$, respectively, and is smaller than $ASV(\hat{\varrho}_n)$ for $\nu \le 16$. We further take note of the remarkable fact that for all uncorrelated $t$- and normal distributions, the asymptotic relative efficiency of Kendall's tau with respect to the respective MLE is above 90\% for $\nu \ge 2$. It is more than 99\% at an uncorrelated $t_{5}$ distribution.

The other popular nonparametric correlation measure, Spearman's rho, which is often considered alongside Kendall's tau, is defined as Pearson's linear correlation of the ranks of the data. It can be written as
\[
	\hat{r}_n = \frac{12}{(n-1)n(n+1)} \sum_{i=1}^{n} R_n(X_i) R_n(Y_i) \ - \ 3\,\frac{n+1}{n-1},
\]
where $R_n(X_i)$ denotes the rank of the $i$th observation $X_i$ among $X_1,\ldots,X_n$, likewise $R_n(Y_i)$.
The population version of Spearman's rho, $s  =  12 E\left(F_X(X) F_Y(Y)\right) - 3$,
is also a function of the copula. Generally, Kendall's tau and Spearman's rho have similar statistical properties. See, e.g., \citet[][Chap.~5]{nelsen:2006} 
for details on their relationship. 
\citet{croux:dehon:2010} 
compare both with respect to robustness and efficiency and arrive at the conclusion, that in both respects their performance is comparable, but Kendall's tau is slightly favorable. \citet*{wied:dehling:vankampen:vogel:2013} propose a nonparametric, robust change-point test for constant correlation for strongly mixing sequences that is closely related to Spearman's rho. They consider the test statistic
\be \label{eq:s_k}
		\hat{T}_{s,n} = \max_{1 \le k \le n} \frac{k}{\sqrt{n}} | \hat{s}_k - \hat{s}_n|,
\quad
\mbox{where }
\quad
	 \hat{s}_k = 12 n^{-3} \sum_{i=1}^k R_n(X_i) R_n(Y_i)  - 3 - 12/n, \quad 1 \le k \le n,
\ee
along with a suitable long-run variance estimator in the same vein as $\hat{\sigma}^2_{\tau,n}$ in (\ref{eq:D.hat}) above.
The proof of its convergence is based on an invariance principle for the multivariate sequential empirical process. 
Despite the mentioned practical parity of Kendall's tau and Spearman's rho, this test has a low efficiency compared to our Kendall's tau based test (see Section \ref{sec:sim.res}). The reason lies in the usage of $R_n(\cdot)$ instead of $R_k(\cdot)$ in (\ref{eq:s_k}). Recently, \citet{kojadinovic:quessy:rohmer:2016} proposed a new constancy test for Spearman's rho which fixes this drawback and has therefore considerably higher power.
Spearman's rho is asymptotically equivalent to a $U$-statistic of order 3, and an asymptotic analysis of the related test statistic
\[
		\hat{T}_{r,n} = \max_{1 \le k \le n} \frac{k}{\sqrt{n}} | \hat{r}_k - \hat{r}_n|
\]
by means of $U$-statistics theory is mathematically much more involved. Since Spearman's rho, on the other hand, exhibits no pronounced advantage over Kendall's tau, we do not pursue this further here. Finally, we note that both estimators require a comparable computing effort. Both can be computed in $O(n \log n)$ time. Simple algorithms to compute the test statistics require $O(n^2)$.


 \section{Simulation Results}
\label{sec:sim.res}

In this section we give some numerical results, comparing the performance of the test to the previous proposals by \citet{wied:kraemer:dehling:2011} and \citet{wied:dehling:vankampen:vogel:2013}. These are based on the test statistics $\hat{T}_{\varrho,n}$ and $\hat{T}_{s,n}$, respectively (cf.~Section~\ref{sec:props}) and referred to as \emph{Pearson test} and \emph{Spearman test} in the following. We call the new proposal \emph{Kendall test}. Moreover, we consider the \emph{improved Spearman test} by \citet{kojadinovic:quessy:rohmer:2016}.

Throughout, we estimate the long-run variance $\sigma^2_\tau$ of the Kendall test by the estimator $\hat{\sigma}^2_{\tau,n}$ given by (\ref{eq:D.hat}), where we choose the bandwidth $b_n = \lfloor 2 n^{1/3}\rfloor$ and $\kappa$ to be the quartic kernel
\[
	\kappa(x) = (1-x^2)^2  \bEins_{[-1,1]}(x).
\]
Altogether we found neither the choice of the kernel nor the bandwidth to be very critical. Generally, choosing the bandwidth too small in strong dependence scenario tends to have a bigger impact than choosing it too large in the case of little or no dependence. 
We recommend the quartic kernel since it is smooth and has more of a flat-top-like behavior than the popular Bartlett kernel, giving more weight to small-lag autocorrelations. The variance estimation for $\hat{T}_{\varrho,n}$ and $\hat{T}_{s,n}$ is done according to the authors' proposals, which both also implement kernel estimators following \citet{dejong:2000}. 
For the improved Spearman test by \citet{kojadinovic:quessy:rohmer:2016}, we use implementation in the R-package npcp \citep{R:npcp}
with the Bartlett kernel for the long-run variance estimation. 

For all tests and all dependence scenarios, we take the bandwidth $b_n = \lfloor 2 n^{1/3} \rfloor$.
This is a relatively large bandwidth, which is suitable for strong dependence as in our AR(1) example below.
A smaller bandwidth would be more appropriate for no or little serial dependence. 
The problem of selecting an optimal, data-adaptive bandwidth is a ubiquitous one. It affects the long-run variance estimation of all change-point tests in a similar way and is analogous to the optimal blocklength selection problem for subsampling or bootstrapping, which may also be employed to obtain critical values. 
We do not discuss the problem of automatic bandwidth selection in the present article. Our concern is how the use of different estimators affects the behavior of the change-point tests. The methods under consideration use the same technique for long-run variance estimation and, hence, choosing the same bandwidth in all cases allows a fair comparison of the methods. 
However, any bandwidth selection procedure previously proposed, see, e.g., \citet[][Section 3.3]{kojadinovic:quessy:rohmer:2016} and the references therein, can also be put to use for the Kendall test and is likely to further improve the power of the test as compared to the results presented below. 

We consider three data models that implement three different types of serial dependence: independent observations, a multivariate AR(1) process and a constant conditional correlation (CCC) GARCH(1,1) process, as considered, e.g., in \citet{aue:2009}. All three models are based on a sequence of independent and identically distributed innovations $(\delta_i,\varepsilon_i)$, $i \in \Z$, having a bivariate, centered elliptical distribution $\Ee_2(\bNull,S)$ with 
\[ 	 S = 
	 \begin{pmatrix}
		 1 & \varrho \\
		 \varrho & 1 
	 \end{pmatrix},
\]
where the parameter $|\varrho| \le 1$ is equal to the usual moment correlation if the second moments of $(\delta_1,\varepsilon_1)$ are finite.
\begin{model} \label{mod:1} \rm
$\displaystyle  \begin{pmatrix}
	X_i \\
	Y_i \\
	\end{pmatrix}
	=
	\begin{pmatrix}
		\delta_i \\
		\varepsilon_i \\
		\end{pmatrix},
		\qquad i \in \Z,$
\end{model}
\begin{model} \label{mod:2} \rm 
The series $((X_i,Y_i))_{i \in \Z}$ follows the AR(1) process
\[
	\begin{pmatrix}
	X_i \\
	Y_i \\
	\end{pmatrix} 
	= 0.8 
		\begin{pmatrix}
	X_{i-1} \\
	Y_{i-1} \\
	\end{pmatrix} + 
	\begin{pmatrix}
		\delta_i \\
		\varepsilon_i \\
		\end{pmatrix},
		\qquad i \in \Z,
\]
with AR parameter $\varphi = 0.8$.
\end{model}
\begin{model} \label{mod:3} \rm 
The series $((X_i,Y_i))_{i \in \Z}$ follows the CCC-GARCH(1,1) process
\[
	\begin{pmatrix}
	X_i \\
	Y_i \\
	\end{pmatrix}
	= 
	\begin{pmatrix}
		\sigma_{i} \\
		\tau_{i} \\
	\end{pmatrix} 
	\circ 
	\begin{pmatrix}
		\delta_i \\
		\varepsilon_i \\
		\end{pmatrix}
	\ \ \mbox{ with } \ \ 
		\begin{pmatrix}
		\sigma_{i}^2 \\
		\tau_{i}^2 \\
	\end{pmatrix} 
	=
		\begin{pmatrix}
		0.1 \\
		0.1 \\
	\end{pmatrix} 
	+
		\begin{pmatrix}
		0.1 \\
		0.1 \\
	\end{pmatrix} 
	\circ
	\begin{pmatrix}
		X_{i-1}^2 \\
		Y_{i-1}^2 \\
	\end{pmatrix} 
			+
		\begin{pmatrix}
		0.84 \\
		0.84 \\
	\end{pmatrix} 
	\circ
	\begin{pmatrix}
		\sigma_{i-1}^2 \\
		\tau_{i-1}^2 \\
	\end{pmatrix},  			
		\quad i \in \Z,
\]
where $\circ$ denotes the Hadamard product, i.e.\ component-wise vector multiplication.
\end{model}

Most simulation results below are for $n = 500$. At this sample size, the distribution of the test statistic $\hat{T}_{\tau,n}/(2\hat\sigma_{\tau,n})$ is well approximated by its limit distribution under the null hypothesis in all dependence scenarios considered.
For the first half of the data, we sample independent realizations $(\delta_i,\varepsilon_i)$, $i = 1,\dots,250$, with correlation parameter $\varrho_1 = 0.4$. For the second half of the data, we use the correlation parameter $\varrho_2$, for which we allow the values $0.4$ (null hypothesis), $0.6$, $0.8$, $0.2$, $0$, $-0.2$, $-0.4$.
Thus in Model \ref{mod:1}, we have a constant correlation of 0.4 at the beginning and then a sudden jump, whereas in Models \ref{mod:2} and \ref{mod:3}, there is a gradual but quick change in the correlation of the observed process $(X_i,Y_i)$. 
Also note that, under the null, the data process $((X_i,Y_i))_{i \in \Z}$ has the same marginal correlation as the innovation process $((\delta_i,\varepsilon_i))_{i \in \Z}$ in Models \ref{mod:1} and \ref{mod:2}, but generally not in Model \ref{mod:3}. 

In Models \ref{mod:1} and \ref{mod:2}, we consider five different elliptical distributions for $(\delta_i,\varepsilon_i)$: the bivariate normal distribution and bivariate $t_\nu$-distribution with $\nu = 20, 5, 3, 1$ The parameter $\nu > 0$ is called the \emph{degrees of freedom} or the \emph{tail index}. The $t_{20}$ distribution has slightly heavier tails than the normal, whereas $t_5$, $t_3$ and $t_1$ serve as examples of very heavy-tailed distributions. The $t_{\nu}$ distribution possesses finite moments of order $\alpha$ for any $\alpha < \nu$. Thus, the Pearson test by \citet{wied:kraemer:dehling:2011}, which requires finite fourth moments, does not work for $\nu =3$ and $\nu =1$. 
From an economic point of view, however, the tail indices 3 and 1 are most interesting. There is evidence that financial returns on many stocks, stock indices and foreign exchange rates in developed economies typically have tail indices $\nu$ in the interval $\nu \in (2,4)$ and thus have finite variances but infinite fourth moments. It has emerged that $\nu=3$ is an appropriate choice for financial returns in developed markets and \citet{ibragimov:ibragimov:kattuman:2013} provide empirical evidence that $\nu$ may be even smaller than $2$ for foreign exchange rates in emerging economies. 

In Model \ref{mod:3}, we used the normal distribution and $t_\nu$ distributions with $\nu = 20, 8, 5$. The CCC-GARCH model generates heavy tails also for normal innovations. For $t_\nu$ innovations with $\nu = 3, 1$, the process explodes. For $\nu = 5$, the CCC-GARCH process of Model \ref{mod:3} generates very pronounced volatility clusters.

For each combination of model, jump height and marginal distribution we generate 1000 samples and compute the three test statistics from each sample. The observed rejection frequencies at the significance level $.05$ for sample size $n=500$ are given in Tables \ref{tab:1}, \ref{tab:2} and \ref{tab:3} for Models \ref{mod:1}, \ref{mod:2} and \ref{mod:3}, respectively.
\begin{table}[t]
\caption{Efficiency comparison of several correlation change-point tests under Model \ref{mod:1}. Different marginal distributions, 500 observations, different jump sizes in the middle of the sample. Empirical rejection frequencies at the asymptotic .05 level based on 1000 repetitions.}
%
\begin{tabular}{c@{\quad}r@{\qquad}r@{\quad \ }r@{\quad \ }r@{\quad \ }r@{\quad \ }r@{\quad \ }r@{\quad \ }r}
\multicolumn{2}{r@{\qquad}}{ Change at $n/2$: } & none & -.2 & +.2 & -.4 & +.4 & -.6 & -.8  \\
\hline\\[-1.0ex]
Distribution & Test & & & & & & \\
\hline\\[-1.0ex]
	normal 		& Pearson 							& .04	& .46 	& .70 	& .97  	& 1.00  &  1.00 &  1.00 \\
				 		& Spearman  						& .04 & .06		& .07		& .22 	&  .20 	&   .47	&   .78	\\
						& improved Spearman  		& .03 & .43 	& .58 	& .95 & 1.00	&  1.00 &  1.00 \\ 
				 		& Kendall 							& .05 & .44		& .65		& .96 	& 1.00  &  1.00 &  1.00 \\[.6ex]
	$t_{20}$ 	& Pearson 							& .04	& .42		& .65 	& .97 	&  1.00 & 1.00  & 1.00  \\
					 	& Spearman  						& .03	& .07		& .08		& .22		&   .20	&  .47  &  .77 	\\
						& improved Spearman  		& .02 & .42 	& .56 	& .94 	&  1.00 & 1.00 	& 1.00 \\ 
					 	& Kendall 							& .04	& .46 	& .63		& .97 	&  1.00 & 1.00  & 1.00 	\\[.6ex]
	$t_{5}$  	& Pearson 							& .04 & .24		& .41		& .73		&		.95 &	 .95	&  .98	\\
					 	& Spearman  						& .04	& .08		& .08		& .22		&  	.20	&  .46	&  .76	\\
						& improved Spearman  		& .03 & .38 	& .50 	& .91 	&  1.00 & 1.00 	& 1.00 \\ 
					 	& Kendall 							& .04 & .41 	& .55  	& .95  	&  1.00 & 1.00  & 1.00  \\[.6ex]
		$t_{3}$	& Pearson 							& .06 & .14		& .25		& .39		&  	.69 &  .64  &  .79	\\
					 	& Spearman  						& .03	& .08		& .08		& .21		&  	.18	&  .43	&	 .72	\\
						& improved Spearman  		& .04 & .32 	& .46 	& .88 	&  1.00 & 1.00 	&   1.00 \\ 
					 	& Kendall 							& .03	& .39 	& .52 	& .91 	&  1.00 & 1.00 	&	1.00  \\[.6ex]
 $t_{1}$		& Pearson 							& .47	&	.48		& .50		&	.49 	&   .56	&  .52	&  .51 	\\
					 	& Spearman  						& .03	& .06		& .07		& .17		&   .17 &	 .38	&  .63 	\\
						& improved Spearman  		& .03 & .25 	& .33 	& .74 	&   .96 &  .98 	& 1.00 \\ 
					 	& Kendall 							& .04	& .29 	& .38 	& .83 	&   .98 & 1.00  & 1.00 	\\[.6ex]		 
\end{tabular} \label{tab:1}
\end{table}
At Table \ref{tab:1} (independence scenario) we note the following.
\begin{enumerate}[\it(1)]
\item 
The Pearson test is slightly better than the Kendall test for the normal distribution. Both tests lose power with increasing tails, but the loss is much smaller for the Kendall test. For the $t_{20}$ distribution, the results are comparable. The Kendall test is clearly better for heavier tails. These observations are fully in line with our expectations considering the efficiency comparison of the respective correlation measures in Section \ref{sec:props}.
\item
Throughout, the naive Spearman test has a very low power, 
which is not true for the improved Spearman test. The improved Spearman test and Kendall test show comparable results under normality as well as heavy tails, with advantages for the Kendall test. This is in line with the efficiency comparison of both estimators by \citet{croux:dehon:2010}.  

\item
For the $t_3$ distribution, the Pearson test yields still approximate results, whereas for the $t_1$ distribution it is completely useless.
\een
\begin{table}[t]
\caption{Efficiency comparison of several correlation change-point tests. $((X_i,Y_i))_{i=1,\ldots,n}$ AR(1) process with AR-parameter $\varphi = 0.8$. Different innovation distributions, 500 observations, several alternatives. Empirical rejection frequencies at the asymptotic .05 level based on 1000 repetitions.}
%
\begin{tabular}{c@{\quad}r@{\qquad}r@{\quad \ }r@{\quad \ }r@{\quad \ }r@{\quad \ }r@{\quad \ }r@{\quad \ }r}
\multicolumn{2}{r@{\qquad}}{ Change at $n/2$: } & none & -.2 & +.2 & -.4 & +.4 & -.6 & -.8  \\
\hline\\[-1.0ex]
Distribution & Test & & & & & & \\
\hline\\[-1.0ex]
	normal 		& Pearson 					& .07	& .13 	& .27		& .45 	& .77 & .78 & .94 \\
				 		& Spearman  				& .05 & .06		& .04		& .07 	& .07 & .11	& .18	\\
						& improved Spearman & .03 & .11 	& .15 	& .39 	& .73 & .73 & .95 \\ 
				 		& Kendall 					& .05 & .14		& .19		& .46 	& .73 & .79 & .96	\\[.6ex]
	$t_{20}$ 	& Pearson 					& .06	& .11		& .31		& .41   & .77	& .77 & .94 \\
					 	& Spearman  				& .05	& .06		& .06		& .08		& .08	& .12 & .20 \\
						& improved Spearman & .03 & .12 	& .15 	& .36 	& .70 & .70 & .95 \\ 
					 	& Kendall 					& .04	& .13		& .23		& .42		& .74 & .79 & .95	\\[.6ex]
	$t_5$  		& Pearson 					& .08 & .12		& .26 	& .34		& .67	&	.67	& .89	\\
					 	& Spearman  				& .06	& .05		& .06		& .08		& .09	& .12	& .18	\\
						& improved Spearman & .03 & .13 	& .15 	& .35 	& .67 & .68 & .90 \\ 
					 	& Kendall 					& .05 & .15		& .19  	& .40 	& .67	& .73 & .95 \\[.6ex]
		$t_3$		& Pearson 					& .10 & .11		& .26		& .25		& .56 & .50 & .67	\\
					 	& Spearman  				& .06	& .08		& .06		& .08		& .09	& .09	&	.17	\\
						& improved Spearman & .05 & .09 	&   .12 &   .28 & .58 & .58 & .84 \\ 
					 	& Kendall 					& .05	& .12		& .18 	& .34		& .62 & .67	&	.90 \\[.6ex]
 $ t_1$			&  Pearson 					& .46	& .49		& .52		& .50 	& .56	& .50	& .54 	\\
					 	& Spearman  				& .08	& .07		& .08		& .09		& .11 &	.12	& .14 	\\
						& improved Spearman & .03 & .05 	& .06 	& .13 	& .20 & .24 & .41 \\ 
					 	& Kendall 					& .06	& .12		& .11		& .18		& .34	& .34 & .53	\\[.6ex]		 
\end{tabular} \label{tab:2}
\end{table}
Analyzing Table \ref{tab:2} (AR scenario) we find that
\begin{enumerate}[\it(1)]
\setcounter{enumi}{3}
\item
the power of all tests is lower for the AR(1) process than in the independent case, and
\item
the observations made at Table \ref{tab:1} concerning the comparison of the tests generally also apply here. 
The Kendall test is slightly better than the improved Spearman test.
\item
For normal, $t_{20}$, $t_5$ and $t_3$ innovations, the performance of the Kendall, the improved Spearman and the Pearson test are rather similar. The effect of the heavy tails is less pronounced than in the independent case. 
This is not entirely surprising. In Model \ref{mod:2}, the marginal distribution of the process can be expressed as a 
sum of independent random variables. Although it is generally not normal, it is, purely heuristically speaking, closer to a normal distribution than the innovations (if these possess finite second moments). 
\een
\begin{table}[t]
\caption{Efficiency comparison of several correlation change-point tests. $((X_i,Y_i))_{i=1,\ldots,n}$ multivariate GARCH process. Different innovation distributions, 500 observations, several alternatives. Empirical rejection frequencies at the asymptotic .05 level based on 1000 repetitions.}
%
\begin{tabular}{c@{\quad}r@{\qquad}r@{\quad \ }r@{\quad \ }r@{\quad \ }r@{\quad \ }r@{\quad \ }r@{\quad \ }r}
\multicolumn{2}{r@{\qquad}}{ Change at $n/2$: } & none & -.2 & +.2 & -.4 & +.4 & -.6 & -.8  \\
\hline\\[-1.0ex]
Distribution & Test & & & & & & \\
\hline\\[-1.0ex]
	normal 		& Pearson 					& .06 &  .47 &  .63 &  .94 &  .99 & 1.00 & 1.00 \\
				 		& Spearman  				& .04 &  .07 &  .09 &  .22 &  .17 &  .47 &  .76 \\
						& improved Spearman & .03 &  .43 &  .56 &  .96 & 1.00 & 1.00 & 1.00 \\ 
				 		& Kendall 	  			& .04 &  .45 &  .59 &  .96 & 1.00 & 1.00 & 1.00 \\[.6ex]
	$t_{20}$ 	& Pearson 					& .07 &  .46 &  .55 &  .90 &  .95 &  .98 & 1.00 \\
					 	& Spearman  				& .04 &  .07 &  .07 &  .22 &  .19 &  .46 &  .78 \\
						& improved Spearman & .02 &  .41 &  .51 &  .94 & 1.00 & 1.00 & 1.00 \\ 
					 	& Kendall 					& .05 &  .46 &  .57 &  .96 & 1.00 & 1.00 & 1.00 \\ [.6ex]
	$t_8$  		& Pearson 					& .13 &  .42 &  .42 &  .76 &  .78 &  .90 &  .94 \\
					 	& Spearman  				& .05 &  .07 &  .08 &  .21 &  .19 &  .42 &  .72 \\ 
						& improved Spearman & .03 &  .38 &  .50 &  .90 & 1.00 & 1.00 & 1.00 \\ 
					 	& Kendall 					& .04 &  .45 &  .52 &  .94 &  .99 & 1.00 & 1.00 \\[.6ex]
		$t_5$		& Pearson 					& .30 &  .44 &  .36 &  .66 &  .62 &  .75 &  .81 \\
					 	& Spearman  				& .06 &  .09 &  .10 &  .20 &  .23 &  .41 &  .69 \\
						& improved Spearman & .03 &  .34 &  .39 &  .88 &  .99 & 1.00 & 1.00 \\ 
					 	& Kendall 					& .07 &  .40 &  .41 &  .89 &  .94 &  .99 &  .99	 \\[.6ex]						
\end{tabular} \label{tab:3}
\end{table}
At Table \ref{tab:3} (CCC-GARCH scenario) we observe that 
\begin{enumerate}[\it(1)]
\setcounter{enumi}{6}
\item
the efficiencies of the test are comparable to the independence case. GARCH processes are white noise with zero autocorrelations. Serial dependence in the second-order characteristics appear to influence the tests less.
\item
The Pearson test has difficulties keeping the size. A size-adjusted power comparison shows a strict superiority of the Kendall test over the Pearson test.
\een

\begin{table}[t]
\caption{Efficiency comparison of several correlation change-point tests at Model \ref{mod:3} (multivariate GARCH).
Innovation correlation changes from $\varrho_1 = 0.4$ to $\varrho_2=0$, normal innovations.
Different sample sizes, different jump locations. Empirical rejection frequencies at the asymptotic .05 level based on 1000 repetitions.}
%
\begin{tabular}{c@{\quad}r@{\qquad}r@{\quad \ }r@{\quad \ }r@{\quad \ }r@{\quad \ }r}
\multicolumn{2}{r@{\qquad}}{ Change at : } & none & $n/8$  & $n/4$ & $3n/8$ & $n/2$   \\
\hline\\[-1.0ex]
Sample size & Test & & & & &  \\
\hline\\[-1.0ex]
	$n=250$ 	& Pearson 					&  .06 & 	 .12 &  .35 &  .59 	&  .64 \\
				 		& Spearman  				&  .03 & 	 .04 &  .07 &  .10 	&  .10 \\ 
						& improved Spearman &  .02 &   .04 &  .20 &  .47 	&  .58 \\  
				 		& Kendall 	  			&  .04 & 	 .11 &  .38 &  .60 	&  .69 \\[.6ex]
	$n=500$ 	& Pearson 					&  .06 & 	 .22 &  .76 &  .92 	&  .94 \\ 
					 	& Spearman  				&  .04 & 	 .06 &  .11 &  .15 	&  .22 \\ 
						& improved Spearman &  .03 &   .12 &  .68 &  .91 	&  .96 \\ 
					 	& Kendall 					&  .04 & 	 .22 &  .79 &  .94 	&  .96 \\[.6ex]			
	$n=1000$  & Pearson 					&  .06 & 	 .55 &  .99 & 1.00 	& 1.00 \\ 
					 	& Spearman  				&  .05 &   .10 &  .26 &  .40 	&  .43 \\ 
						& improved Spearman &  .03 &   .44 &  .99 & 1.00 	& 1.00 \\ 
					 	& Kendall 					&  .04 & 	 .56 & 1.00 & 1.00 	& 1.00 \\ [.6ex]
\end{tabular} \label{tab:4}
\end{table}
To give an impression of the power of the test in other data situations, we also include some limited simulation results for sample sizes $n = 250, 500, 1000$ and change locations ranging from $n/8$ to $n/2$. The data generating process is CCC-GARCH (Model \ref{mod:3}) with normal innovations and $\rho_1 = 0.4$ and $\rho_2 = 0$. The parameters for the long-run variance estimation are as before. 
The results are summarized in Table \ref{tab:4}. We find the general picture mediated by Table \ref{tab:3} concerning the comparison of the tests  confirmed. The power of all change-point tests is lower for changes occurring nearer to either end of the data sequence.

Altogether the simulation results are favorable for the Kendall test. It is clearly superior  to the Pearson test at heavy-tailed data and qualitatively non-inferior under normality. It is non-inferior to the improved Spearman test in all scenarios considered.
\citet{kojadinovic:quessy:rohmer:2016} also propose a bootstrapping procedure to obtain critical values for the improved Spearman test, which their simulation results suggest improves the power slightly. A similar procedure applicable to the Kendall test has been proposed recently by \citet{buecher:kojadinovic:2016}.

The above data generating processes are all based on elliptical innovations. 
Any change in the generalized correlation coefficient $\varrho$ constitutes a shift of the same magnitude for all dependence measures considered, and the corresponding tests generally behave similarly. 
However, in data situations where a change occurs that affects the various dependence measures to a different degree, the tests may behave qualitatively differently. 
A class of examples where a change occurs in the Pearson correlation but neither in Kendall's tau nor Spearman's rho is generated by leaving the copula constant but changing the marginal distributions. 
Similarly, one may change the copula in such a way that Kendall's tau or Spearman's rho remain fixed. (A change-point test for the whole copula has recently been proposed by \citet{bucher:2014}). To illustrate this point, consider the following example: Suppose in the first half of the sequence that the observations are i.i.d., following the distribution of
\[ 		
	\begin{pmatrix}
	X \\
	Y \\
 \end{pmatrix}  
  =
  N_2\left(
			\begin{pmatrix}
			0 \\
			0 \\
		 \end{pmatrix}, 
		\begin{pmatrix}
		 1 & 0.6 \\
		 0.6 & 1 		 	 
	\end{pmatrix}		
	\right).
\]
In the second half, the observations are i.i.d.\ with distribution of $(X,Y^3|Y|/\sqrt{105})$. Thus, before and after the change, the data have mean zero, marginal variances one, and a Kendall's tau coefficient of $\tau = 2 \pi^{-1} \arcsin(0.6) \approx 0.41$. However, the Pearson correlation changes from 0.6 to $\sqrt{2/(105 \pi)}24/5 \approx 0.37$. Consequently, the Kendall test has no power against this change whereas the Pearson test does. 
If the distribution of the observations changes from that of $(X,Y)$ in the first half to that of $(U,V)$ in the second half, where $U = 2\sqrt{3}(\tilde{U}-1/2)$ and $V = 42\sqrt{3}/13(\tilde{U}^{13}-1/14)$ for $\tilde{U} \sim U[0,1]$, then, as before, mean and marginal variances remain the same, also the Pearson correlation remains at 0.6, but the Kendall's tau coefficient changes from 0.41 to 1. In this case, the Pearson test has no power, whereas the Kendall test has. In Table \ref{tab:inconsistency}, we give simulated rejection frequencies (at the 0.05 level) for the two examples for $n=500$ with the settings for the long-run variance estimation as before.
\begin{table}[t]
\caption{Empirical rejection frequencies at the asymptotic .05 level based on 1000 repetitions; Sample size $n = 500$.}
%
\begin{tabular}{r@{\quad}|@{\quad}r@{\qquad}r} 
															&  Kendall test & Pearson test \\
\hline\\[-1.0ex]
 constant copula example    	&   .23 	&   .94  \\
 constant correlation example &	 1.00   &   .07	\\
\end{tabular} \label{tab:inconsistency}
\end{table}
Generally, the differences in the dependence measures tend to be of comparable size in realistic data models and not as pronounced as in the artificial examples above.

 \section{Data Examples}
\label{sec:example}
\begin{figure}[ht]
\centering
\includegraphics[width=0.95\textwidth]{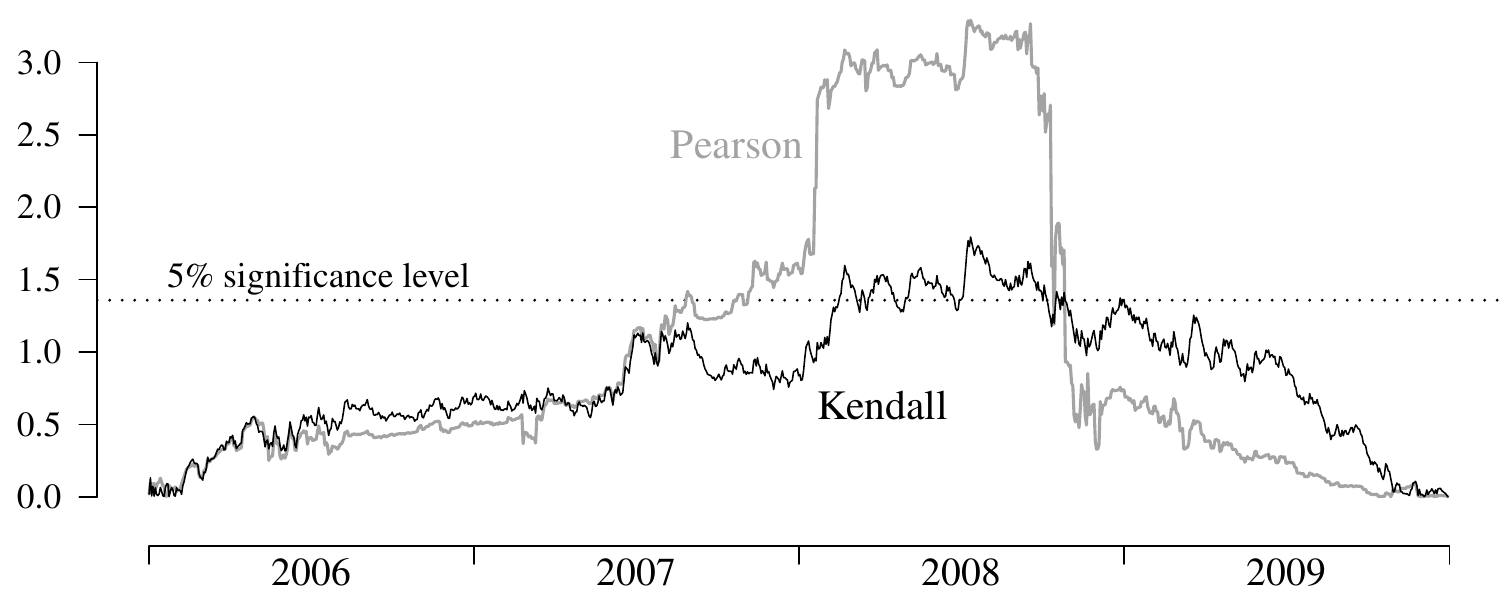}
\caption{
	Processes  $\frac{k}{\sqrt{n}} | \hat{r}_k - \hat{r}_n |$ (grey) and 
	$\frac{k}{\sqrt{n}} | \hat{\tau}_k - \hat{\tau}_n|$ (black), $k=1,\dots,n$, computed from 
	log returns of DAX and S\&P 500 between Jan 1, 2006, and Dec 31, 2009.
} \label{fig:3}
\end{figure}
\citet{wied:kraemer:dehling:2011} analyze the dependence between the German stock index (DAX) and the Standard and Poor's 500 (S\&P 500). We apply the Kendall test and the Pearson test (with the same parameter choices for the variance estimation as in the simulation section) to the daily log returns of the two financial indices in the years 2006 through 2009 (1043 observations). The second half of this period covers what has been termed the Global Financial Crisis. The processes $\big(\frac{k}{\sqrt{n}} | \hat{r}_k - \hat{r}_n | \big)_{k=1,\dots,n}$ and $\big(\frac{k}{\sqrt{n}} | \hat{\tau}_k - \hat{\tau}_n| \big)_{k=1,\dots,n}$ are depicted in Figure~\ref{fig:3}. Their maxima are the values of the test statistics of the Pearson and the Kendall test, respectively. Both tests give a p-value below 0.005, and both attain their maximum on July 14, 2008, at the height of the financial crisis. (Lehman Brothers filed for bankruptcy on September 14, 2008.) The tests behave similarly and their outcome supports the assumption that the dependence between both indices considered can not be assumed to be identical before and during the financial crisis of 2008. 

With a second set of data, the differences between both tests become apparent. We consider the Dow Jones Industrial Average and the Nasdaq Composite in the years 1987 and 1988 (Figure~\ref{fig:4}). The most notable feature of both time series is the heavy loss on October 19, 1987, commonly known as Black Monday. Here we may ask in particular the question if the market conditions substantially changed after this date. Does Black Monday constitute a break in the correlation between the two time series? The Pearson test reports a p-value indistinguishable from zero by machine accuracy.
The underlying processes of the Pearson and the Kendall test are shown in Figure~\ref{fig:5}. The outcome of the Pearson test is determined by the peak on October 19, 1987, which is explained as follows. On October 19, both indices suffered heavy losses, suggesting a strong positive correlation of their log returns. The following day the Dow Jones recovered to some small degree, whereas the Nasdaq experienced an even larger drop, suggesting strong negative correlation. Thus the process of successive sample correlations jumps up and immediately down again. 
\begin{figure}[t]
\centering
\includegraphics[width=0.95\textwidth]{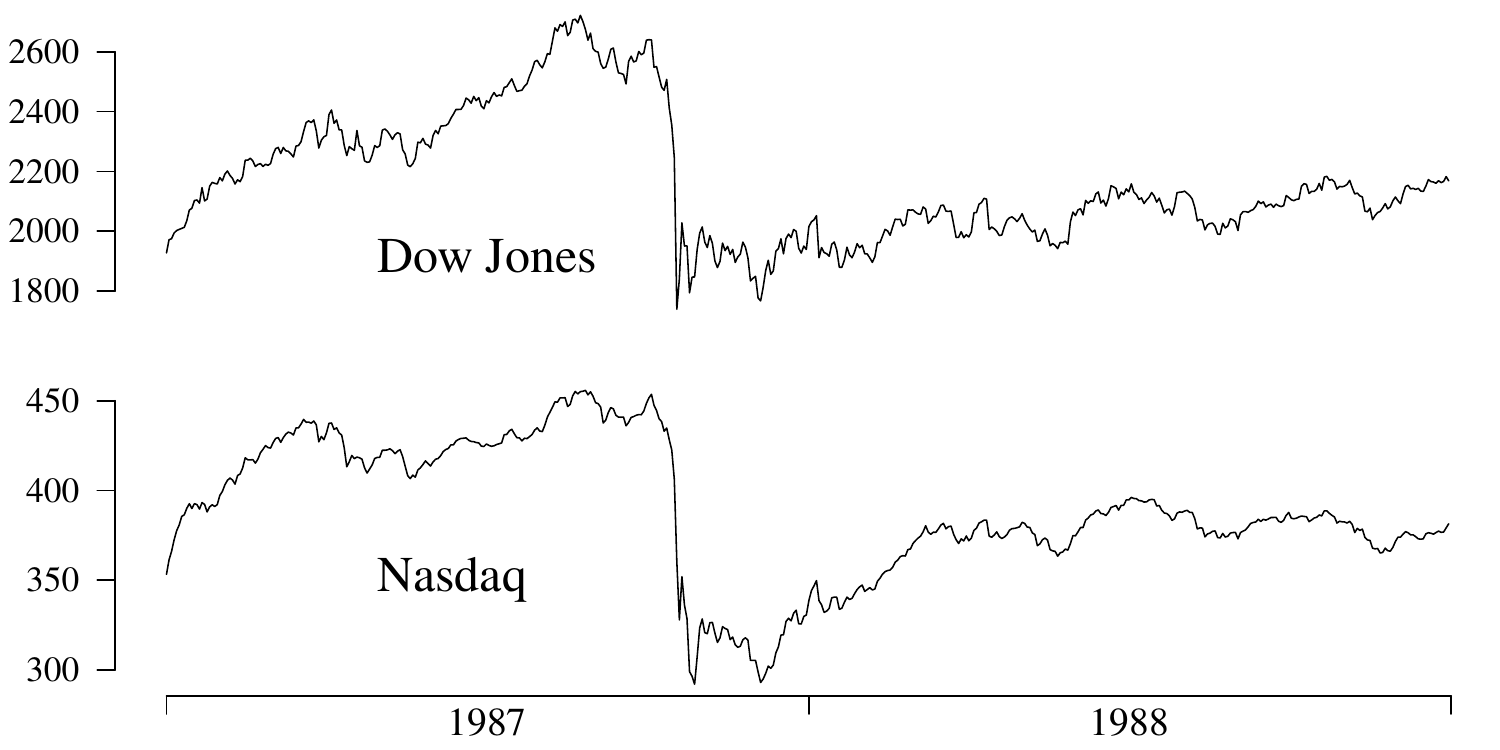}
\caption{
	Daily closings of Dow Jones Industrial Average and Nasdaq Composite from Jan 1, 1987, to Dec 31, 1988.
}
 \label{fig:4}
\end{figure}
\nopagebreak
\begin{figure}[t]
\centering
\includegraphics[width=0.95\textwidth]{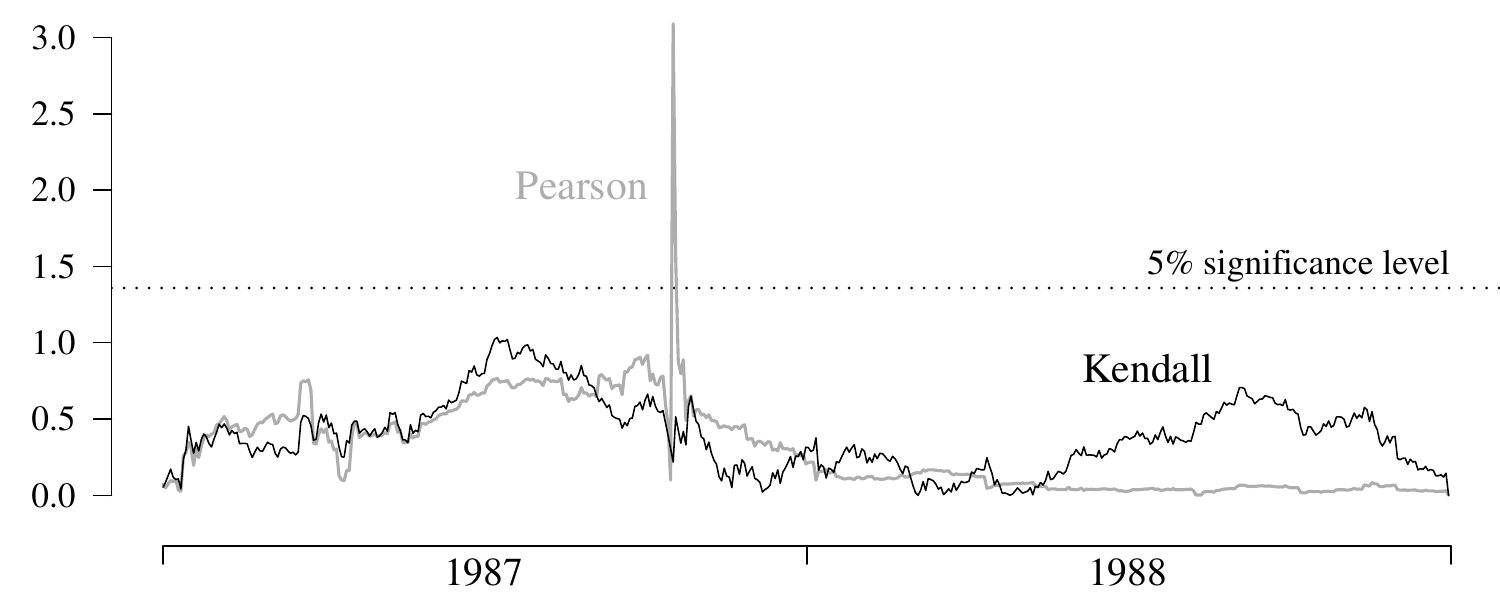}
\caption{
	Processes  $\frac{k}{\sqrt{n}} | \hat{r}_k - \hat{r}_n |$(grey) and 
	$\frac{k}{\sqrt{n}} | \hat{\tau}_k - \hat{\tau}_n|$ (black), $k=1,\dots,n$, computed from 
	log returns of DJIA and Nasdaq Composite between Jan 1, 1987, and Dec 31, 1988.
} \label{fig:5}
\end{figure}

The Kendall test gives a p-value of 0.24, indicating one can assume the correlation between the two time series to be constant over the observed time period. The empirical Kendall's tau is 0.52 prior to Black Monday, and 0.56 afterwards. Indeed, the market conditions turned out to be not much different from before, the DJIA even closed positive for 1987. 

The strong impact of a few or even a single extreme observation on the Pearson test is in line with \citet[][Ch.2]{stock:watson:2010}, who illustrate the inappropriateness of normal distribution assumptions and the necessity of considering heavy-tailed distributions in modeling financial time series such as the DJIA during the Black Monday crisis.

\section{Conclusion}

We have presented a fluctuation test for detecting changes in the dependence between two time series based on Kendall's rank correlation coefficient. We have demonstrated the non-inferiority of the test in terms of efficiency and the clear superiority in terms of robustness and applicability to a similar, previously proposed test, which is based on Pearson's moment correlation. To allow arbitrarily heavy-tailed data and very weak assumptions concerning the serial dependence, we have introduced the concept of near epoch dependence in probability. We have studied the asymptotic behavior of the test statistic under stationarity by means of limit theorems for $U$-statistics for weakly dependent, stationary processes.

We showed simulation results for various dependence scenarios with elliptical innovations. The elliptical model is of wide-spread use and its intrinsic symmetry allows a meaningful comparison of the two change-point tests. Outside ellipticity, where \emph{monotone} dependence (which is measured, e.g., by Kendall's $\tau$) and \emph{linear} dependence (which is measured by the Pearson correlation) are not necessarily equivalent, it is a matter of debate which of the two types of dependence is more relevant. A detailed discussion goes beyond the scope of this paper, but it appears that one is often interested in monotone dependence rather than linear dependence, and the prevalent use of the sample correlation to measure monotone dependence is presumably due to its simplicity and historical dominance.

Furthermore, simulations also show that the proposed test possesses a variety of advantageous features that have not been discussed in this paper. It has power against gradual or fluctuating changes in the correlation, not only sudden jumps, as presented in Section \ref{sec:sim.res}. It also exhibits a much better robustness against heteroscedasticity than the Pearson test. However, a thorough theoretical assessment of these properties as well as constructing tests that explicitly allow heteroscedasticity require the study of $U$-statistics at non-stationary sequences, which is mathematically rather involved, as our study of the one-change-point model \ref{mod:la} illustrates.

Further future research directions include, e.g., the extension to more than two dimensions or guidelines for an on-line application of the test with results about the detection time of a change. Another interesting research question, which is related to the one studied here, is to devise a robust test for detecting changes in the \emph{coherence} of two time series. For example, a series of i.i.d.\ variables, shifted by one observation, is highly coherent to the original series, but our test, which only compares observations at the same time point, does not detect that type of dependence.

\small
\appendix
\section*{Appendix}
The Appendix is split into three parts. Appendix \ref{sec:app:pned} provides background on the concept of $P$-near epoch dependence. 
In Appendix \ref{sec:app:local}, we study the asymptotic behavior of the $U$-statistic process under local alternatives, which provides the basis for the proof of Theorem \ref{th:localpower}. Appendix \ref{sec:app:proofs} contains all proofs of the results of the main text. Results labeled D.1, D.2, $\ldots$ refer to the online supplement (Appendix D), available at Cambridge Journals Online (journals.cambridge.org/ect).


\section{$P$-NED: Near epoch dependence in probability}
\label{sec:app:pned}

In this section we add some background on the newly introduced $P$-near epoch dependence and connect it to the usual $L_p$-near epoch dependence (Lemma \ref{lem:l_0}). Similar conditions that embody the idea of approximating functionals in a probability sense are $S$-mixing considered by \citet*{berkes:hoermann:schauer:2009} and the $L_0$-approximability
of \citet[][Chapter~6]{poetscher:prucha:1997}. 

%
First note that the $P$-NED condition (Definition \ref{def:dep} (ii)) is equivalent to convergence in probability of $f_k(\bZ_{-k},\ldots,\bZ_{k})$ to $\bX_0$ for $k \to \infty$. If the latter is true, i.e., if
\[
	\varepsilon_k = \inf
	\left\{ \varepsilon \phantom{\Big|} \middle| \,
		 P \left( \left|  \bX_0 - f_k(\bZ_{-k},\ldots,\bZ_{k})  \right|_1 > \varepsilon\right)  \le \varepsilon
	\right\} \ \to \ 0 \qquad (k \to \infty),
\] 
then setting $a_k = \varepsilon_k$ and 
\[
	\Phi(\varepsilon) = 
	\left( 
		\sup_{l \le k} \frac{P( |\bX_0 - f_l(\bZ_{-l},\ldots,\bZ_{l}) |_1 > \varepsilon_k)}{\varepsilon_l}
	\right) \vee 1
	\qquad \mbox{for } \varepsilon \in [\varepsilon_k,\varepsilon_{k-1})
\]
fulfills (\ref{eq:l_0}). The requirement of the bound on $P \left( \left|  \bX_0 - f_k(\bZ_{-k},\ldots,\bZ_{k})  \right|_1 > \varepsilon\right)$ in (\ref{eq:l_0}) to factorize into an $\varepsilon$-part and a $k$-part is not a restriction, but facilitates rate computations.

Recall that a process $(\bX_n)_{n \in \Z}$ is called \emph{$L_p$ near epoch dependent} ($L_p$-NED), $p \ge 1$, on the process $(\bZ_n)_{n \in\Z}$ if the approximating constants 
$a_{p,k} =  ( E \left|  \bX_0 - E(\bX_0| \F_{-k}^{k})  \right|_p^p )^{1/p}$,   $k \ge 1$,
converge to zero as $k \to \infty$.

Compared to $L_2$-NED, the $P$-NED condition substantially enlarges the class of processes for which the condition is easily checked by many heavy-tailed distributions. For example, the autoregressive process
\[
	\bX_n = \sum_{k=0}^\infty a^k \bZ_{n-k},
\] 
where $a\in(-1,1)$, and $(\bZ_n)_{n\in\Z}$ is an i.i.d.\ sequence of $\R^d$-valued random variables, can easily be seen to be $P$-NED on $(\bZ_n)_{n\in\Z}$ with $a_k=|a|^{\alpha k}$ and $\Phi(\varepsilon) = K\varepsilon^{-\alpha}$ for some constant $K>0$, as long as $P(|\bZ_n|_1\geq t) = O( t^{-\alpha})$ is satisfied. 
This is a very weak condition on the innovation distribution, that should be compared to analogous conditions for an AR(1) process to be mixing.
In particular, the existence of a density is not required, and all standard examples of discrete or heavy-tailed (e.g.~Pareto, Cauchy, geometric) distributions are permitted. 
Furthermore, $P$-NED substantially enlarges the class of Lipschitz functionals considered by \citet[][p.~74]{denker:keller:1986}. The above example would satisfy their condition only if $\bZ_n$ was bounded.
To further detail this example, let, e.g., $a = 1/2$ and $\bZ_i = (Z^{(1)}_i, Z^{(2)}_i)$ have the following discrete distribution
\[
	\textstyle
		P\left( (Z^{(1)}_i,Z^{(2)}_i) = (0,\frac{1}{2}) \right) \ =\ P\left( (Z^{(1)}_i,Z^{(2)}_i) = (\frac{1}{2},0) \right) \ =\ (1-\varrho)/4,
\]\[
		\textstyle
		P\left( (Z^{(1)}_i,Z^{(2)}_i) = (0,0) \right) \ =\ P\left( (Z^{(1)}_i,Z^{(2)}_i) = (\frac{1}{2},\frac{1}{2}) \right) \ =\ (1+\varrho)/4.
\]
The parameter $\varrho \in [-1,1]$ is the moment correlation of this distribution. The Kendall's tau coefficient of $(Z^{(1)}_i,Z^{(2)}_i)$ is $\tau = \varrho/2$. The process $(\bX_i)_{i\in\Z} = ((X^{(1)}_i,X^{(2)}_i))_{i\in\Z}$ is not strongly mixing (since $\bX_{i-1}$ is a deterministic function of $\bX_i$) 
and hence not absolutely regular, but  it is $P$-NED on $(\bZ_i)_{i \in \Z}$ with exponentially decreasing approximation coefficients. The distribution of $(X^{(1)}_i,X^{(2)}_i)$ has the same moment correlation $\varrho$ as $(Z^{(1)}_i,Z^{(2)}_i)$, its Kendall rank correlation is $\tau = 2 \varrho/ (3- \varrho^2)$, and the margins $X^{(1)}_i$ and $X^{(2)}_i$ are uniformly distributed on $(0,1)$.
Then consider the process $((\tilde{X}_i,\tilde{Y}_i))_{i \in \Z}$ with $\tilde{X}^{(1)}_i = H(X^{(1)}_i)$, $\tilde{X}^{(2)}_i = H(X^{(2)}_i)$,
where $H$ denotes the quantile function of a Pareto (type I) distribution with shape parameter $1/2$ and location parameter 1, i.e., $H(x) = (1-x)^{-2}$, 
$x \in [0,1)$. This strictly increasing transformation leaves the $P$-NED coefficients as well as Kendall's tau unchanged. 
The margins $\tilde{X}^{(1)}_i$, $\tilde{X}^{(2)}_i$ are Pareto distributed and have finite moments only up to order less than 1/2. The stationary process 
$( (\tilde{X}^{(1)}_i, \tilde{X}^{(2)}_i))_{i\in\Z}$ is neither mixing nor $L_p$-NED for any $p \ge 1$, but fulfills the assumptions of Corollary \ref{th:change+variance-asy}.
The next lemma connects $P$-NED and $L_p$-NED.
\begin{lemma} \label{lem:l_0}
Let $((\bX_n,\bZ_n))_{n \in \Z}$ be as in Definition \ref{def:dep}.
\begin{enumerate}[(i)]
\item \label{lem:l_0(i)}
If $(\bX_n)_{n\in\Z}$ is $P$-NED on $(\bZ_n)_{n \in \Z}$ with functions $\Phi$ and $f_k$, $k \in \N$, and approximating constants $(a_{k})_{k \in \N}$, and $g: \R^r \to \R^d$ is a Lipschitz continuous function with Lipschitz constant $L$, then the process $(g(\bX_n))_{n \in \Z}$ is $P$-NED on $(\bZ_n)_{n \in \Z}$ with functions $\tilde{\Phi}(\varepsilon) = \Phi(\varepsilon/L)$ and $g \circ f_k$, $k \in \N$, and the same approximating constants $(a_{k})_{k \in \N}$.
\item \label{lem:l_0(ii)}
Let $(\bX_n)_{n\in\Z}$ be bounded and $P$-NED on $(\bZ_n)_{n \in \Z}$ with functions $\Phi$ and $f_k$, $k \in \N$, and approximating constants $(a_{k})_{k \in \N}$. Then $(\bX_n)_{n \in \Z}$ is $L_p$-NED on $(\bZ_n)_{n \in \Z}$ for any $p \ge 1$. If there is further a sequence $(s_k)_{k \in \N}$ of non-negative numbers such that
\be \label{eq:ned.con}
	a_{k} \Phi(s_k) = O(s_k) \qquad \quad  (k \to \infty),
\ee 
then the $L_p$-NED approximating constants $(a_{p,k})_{k \in \N}$ of $(\bX_n)_{n \in \Z}$ satisfy $a_{p,k}^p = O(s_k)$ for $k \to \infty$.
\item \label{lem:l_0(iii)}
Let $(\bX_n)_{n \in \Z}$ be $L_p$-NED, $p \ge 1$, on $(\bZ_n)_{n \in \Z}$ with approximating constants $(a_{p,k})_{k \in \N}$. Then $(\bX_n)_{n \in \Z}$ is $P$-NED on $(\bZ_n)_{n \in \Z}$. If there is further a non-increasing function $\Phi: (0,\infty) \to (0,\infty)$ and a sequence $(s_{k})_{k \in \N}$ of non-negative numbers converging to zero that satisfy
\[
	\Phi(\varepsilon) s_{k} \ge q^{p-1} \, a_{p,k}^p \,\varepsilon^{-p},
\]
then $(\bX_n)_{n \in \Z}$ is $P$-NED on $(\bZ_n)_{n \in \Z}$ with approximation constants $(s_{k})_{k \in \N}$ and function $\Phi$. The functions $f_k$ can be chosen as $f_k(\bZ_{-k},\ldots,\bZ_k) = E(\bX_0|\F_{-k}^k)$, $k \in \N$.
\end{enumerate}
\end{lemma}
Since $\Phi$ is non-increasing, condition (\ref{eq:ned.con}) puts an upper bound on the speed of decay of  $(s_k)_{k \in \N}$ in the sense that, if (\ref{eq:ned.con}) is fulfilled by some sequence $(s_k)_{k \in \N}$, then it is also fulfilled by any sequence $(\tilde{s}_k)_{k \in \N}$ for which $\tilde{s}_k \le s_k$ for all $k$ larger than some $n \in \N$.

The next lemma shows that the near epoch dependence is preserved under transformations that satisfy the variation condition.
\begin{lemma}\label{lem:neu1} Let $(\bX_n)_{n\in\Z}$ be $P$-NED on $(\bZ_n)_{n\in\Z}$ with $a_k\Phi(s_k)=O(s_k^{(2+\delta)/\delta})$. Furthermore, let $g$ be a $U$-statistic kernel with uniform $(2+\delta)$ moments for some $\delta>0$, i.e., $E|g(\bX_0,\bX_n)|^{2+\delta} < M$ for all $n \in\N$, and let $g$ satisfy the variation condition with respect to the distribution $F$ of $\bX_0$ (Assumption \ref{ass1}). Then the sequence  $(g_1(\bX_n))_{n\in\Z}$ is $L_2$-NED on $(\bZ_n)_{n\in\N}$ with approximation constants $a_{k,2}=O(s_k^{\delta/(2+2\delta)})$.
\end{lemma}


\section{$U$-statistics process under local alternatives}
\label{sec:app:local}

In this section, we provide the basis for the proof of Theorem \ref{th:localpower} in Section \ref{sec:app:proofs}.
We analyze the $U$-statistic process $\sum_{1\leq i<j\leq [ns]} g(X_i,X_j)$, $0\leq s\leq 1$, in the case of a change point at time $[n\lambda^\ast ]$, and calculate the asymptotic distribution under local alternatives. We consider the following model
\[
 X_i=X_i^{(n)}=\left\{ 
 \begin{array}{ll}
 \xi_i &\; 1\leq i \leq [n\lambda^\ast]\\
 \xi_i^{(n)} &\;  [n\lambda^\ast]<i\leq n,
 \end{array}
 \right.
\]
where $(\xi_i)_{i\geq 1}$ and $(\xi_i^{(n)})_{i\geq 1}$, $n\geq 1$, are stationary processes
such that  for any $n\geq 1$  the bivariate processes $(\xi_{i},\xi_{i}^{(n)})_{i\geq 1}$ are  P-NED on an absolutely regular process, with mixing coefficients and NED-coefficients independent of $n$.

For the corresponding $U$-statistic process $ \sum_{1\leq i<j\leq [ns]} g(X_i^{(n)},X_j^{(n)})$, we have to distinguish the cases $s\leq \lambda^\ast$ and $s\geq\lambda^\ast$. When 
$s\leq \lambda^\ast$, we have
\[
 \sum_{1\leq i<j\leq [ns]} g(X_i^{(n)},X_j^{(n)}) = \sum_{1\leq i<j\leq [ns]} g(\xi_i,\xi_j),
\]
while for $s\geq \lambda^\ast$ we obtain
\begin{eqnarray}
 && \sum_{1\leq i<j\leq [ns]} g(X_i^{(n)},X_j^{(n)}) \nonumber \\
 && \quad = \sum_{1\leq i<j\leq [n\lambda^\ast ]} g(\xi_i,\xi_j)
 + \sum_{j=[n\lambda^\ast]+1}^{[ns]} \sum_{i=1}^{[n\lambda^\ast]} g(\xi_i,\xi_j^{(n)})+\sum_{j=[n\lambda^\ast]+1}^{[ns]}
  \sum_{i=[n\lambda^\ast]+1}^{j-1} g(\xi_i^{(n)},\xi_j^{(n)}).
  \label{eq:basic-1}
\end{eqnarray}
Note that of the three terms on the r.h.s., the first and the last term are one-sample U-statistics, while the second term is a two-sample $U$-statistic.  We now define the constant terms and the first order terms of the Hoeffding decompositions,
\begin{align*}
\theta   				= &\ E(g(\xi,\eta)),  					& &   \\
\theta_1^{(n)}	= &\ E(g(\xi^{(n)},\eta)),  	& \theta_2^{(n)}& = E(g(\xi^{(n)},\eta^{(n)})), \\ 
g_1(x)					= &\ E(g(x,\xi_j)), 					& g_1^{(n)}(x)	& = E(g(x,\xi_j^{(n)})). 
\end{align*}
Here, $\xi$, $\eta$, $\xi^{(n)}$ and $\eta^{(n)}$ 
are independent random variables such that $\xi$ and $\eta$ have the same distribution as 
$\xi_1$, while $\xi^{(n)}$ and $\eta^{(n)}$ have the same distribution as $\xi_1^{(n)}$.  Note that the functions $g_1$ and $g_1^{(n)}$ are not centered.  In fact,  they require different centerings for each of the three terms on the r.h.s.\ of \eqref{eq:basic-1}.

In this way, we obtain the following Hoeffding decompositions of the summands in \eqref{eq:basic-1}:
\begin{eqnarray*}
g(\xi_i,\xi_j)&=&\theta+(g_1(\xi_i)-\theta)+(g_1(\xi_j)-\theta) +h(\xi_i,\xi_j)\\
g(\xi_i,\xi_j^{(n)})&=&\theta_1^{(n)}+(g_1^{(n)}(\xi_i)-\theta_1^{(n)})+(g_1(\xi_j^{(n)})-\theta_1^{(n)})+h_1^{(n)}(\xi_i,\xi_j^{(n)})\\
g(\xi_i^{(n)},\xi_j^{(n)})&=&\theta_2^{(n)}+(g_1^{(n)}(\xi_i^{(n)})-\theta_2^{(n)} )+
(g_1^{(n)}(\xi_j^{(n)}) -\theta_2^{(n)})+h_2^{(n)}(\xi_i^{(n)},\xi_j^{(n)}).
\end{eqnarray*}
The functions $h$, $h_1^{(n)}$, and $h_2^{(n)}$ are defined in the obvious way, and a straightforward calculation shows that they are degenerate, i.e. that the integral with respect to one of the arguments vanishes.

Plugging this into \eqref{eq:basic-1}, we obtain the Hoeffding decomposition of the $U$-statistic 
\begin{eqnarray*}
&&\hspace{-15mm} \sum_{1\leq i<j\leq [ns]} g(X_i^{(n)},X_j^{(n)}) \ = \ \binom{[n\lambda^\ast]}{2} \theta +[n\lambda^\ast]([ns]-[n\lambda^\ast]) \theta_1^{(n)}
+\binom{[ns]-[n\lambda^\ast]}{2} \theta_2^{(n)} \\
&&\; + ([n\lambda^\ast]-1)\sum_{i=1}^{[n\lambda^\ast]} (g_1(\xi_i)-\theta)
+ ([ns]-[n\lambda^\ast])\sum_{i=1}^{[n\lambda^\ast]} (g_1^{(n)}(\xi_i) -\theta_1^{(n)})  \\
&&\; +[n\lambda^\ast] \sum_{i=[n\lambda^\ast]+1}^{[ns]} (g_1(\xi_i^{(n)}) -\theta_1^{(n)}) 
 +([ns]-[n\lambda^\ast]-1) \sum_{i=[n\lambda^\ast]+1}^{[ns]}(g_1^{(n)}(\xi_i^{(n)})-\theta_2^{(n)}) \\
 &&\; + \sum_{1\leq i<j\leq [n\lambda^\ast]} h(\xi_i,\xi_j)
 + \sum_{j=[n\lambda^\ast]+1}^{[ns]} \sum_{i=1}^{[n\lambda^\ast]} h_1^{(n)} (\xi_i,\xi_j^{(n)})+\sum_{j=[n\lambda^\ast]+1}^{[ns]}
  \sum_{i=[n\lambda^\ast]+1}^{j-1} h_2^{(n)}(\xi_i^{(n)},\xi_j^{(n)}).
\end{eqnarray*}

\vspace{3mm}
We make a number of assumptions regarding the asymptotic behaviour of the terms in the Hoeffding expansion that have to be checked in particular examples. 
\begin{assumption} \label{ass:app1}
There are real numbers $c_1, c_2$ such that $\sqrt{n}(\theta_1^{(n)} -\theta) \rightarrow c_1$ and $\sqrt{n}(\theta_2^{(n)} -\theta)
\rightarrow c_2$.
\end{assumption}

\begin{assumption}\label{ass:app2}
\[
\max_{1\leq k\leq n} \left|\frac{1}{\sqrt{n}}\sum_{i=1}^k\left((g_1(\xi_i)-\theta)-(g_1^{(n)}(\xi_i)-\theta_1^{(n)})\right) \right| \rightarrow 0,
\]\[
\max_{1\leq k\leq n} \left|\frac{1}{\sqrt{n}}\sum_{i=1}^k\left((g_1(\xi_i)-\theta)-(g_1(\xi_i^{(n)})-\theta_1^{(n)})\right) \right|  \rightarrow 0,\]\[
\max_{1\leq k\leq n} \left|\frac{1}{\sqrt{n}}\sum_{i=1}^k\left((g_1(\xi_i)-\theta)-(g_1^{(n)}(\xi_i^{(n)})-\theta_2^{(n)})\right) \right|  \rightarrow 0.
\]
\end{assumption}
\begin{theorem} \label{th:app1}
Under Assumptions \ref{ass:app1} and \ref{ass:app2}, 
\[
  \frac{2}{\sqrt{n}([nt]-1)} \sum_{1\leq i<j\leq [ns]} (g(X_i^{(n)},X_j^{(n)}) -\theta)\; 
   \claw\; 2\sigma W + \phi_{\lambda^\ast}(s),
\] 
where $\sigma$ and $W$ are defined as in Theorem~2.5, and where the function $\phi_\lambda^\ast$ is defined as follows:
\[
  \phi_{\lambda^\ast}(s)=\left\{ 
  \begin{array}{ll}
   \frac{2\lambda^\ast (s-\lambda^\ast)}{s} c_1+\frac{(s-\lambda^\ast)^2}{s} c_2 & s\geq \lambda^\ast, \\
   0 & s\leq \lambda^\ast
  \end{array}
  \right.
\]
with $c_1$ and $c_2$ being defined in Assumption \ref{ass:app1}.
\end{theorem}
\begin{proof}[Proof of Theorem \ref{th:app1}]
We first consider the Hoeffding decomposition of the $U$-statistic under the null hypothesis, i.e.
\begin{eqnarray*}
\sum_{1\leq i<j \leq [ns]} \left(g(\xi_i,\xi_j)-\theta \right) =([ns]-1)\sum_{i=1}^{[ns]} (g_1(\xi_i)-\theta)
 +\sum_{1\leq i<j\leq [ns]} h(\xi_i,\xi_j).
\end{eqnarray*}
Comparing this with the Hoeffding decomposition under the local alternative, we obtain
for $s\geq \lambda^\ast$,
\[
 \sum_{1\leq i<j\leq [ns]} \big( g(X_i^{(n)},X_j^{(n)})-\theta  \big)
- \sum_{1\leq i<j \leq [ns]} \big(g(\xi_i,\xi_j)-\theta \big)
\] \[
\quad = [n\lambda^\ast]([ns]-[n\lambda^\ast]) (\theta_1^{(n)}-\theta)
+\frac{1}{2} ([ns]-[n\lambda^\ast])([ns]-[n\lambda^\ast]-1) (\theta_2^{(n)}-\theta) 
\]
\[ 
\qquad + ([ns]-[n\lambda^\ast]) \sum_{i=1}^{[n\lambda^\ast]} \big((g_1^{(n)}(\xi_i)-\theta_1^{(n)}) 
-(g_1(\xi_i)-\theta)  \big)\] \[
 \qquad +([ns]-[n\lambda^\ast]-1) \sum_{i=[n\lambda^\ast]+1}^{[ns]} 
 \big( (g_1(\xi_i^{(n)}-\theta_1^{(n)})-(g_1(\xi_i)-\theta)  \big)
\] \[
 \qquad + ([ns]-[n\lambda^\ast]-1)\sum_{i=[n\lambda^\ast]+1}^{[ns]} 
 \big( (g_1^{(n)}(\xi_i^{(n)})-\theta_2^{(n)})-(g(\xi_i)-\theta)  \big) 
\] \[
 \qquad  + \sum_{j=[n\lambda^\ast]+1}^{[ns]} \sum_{i=1}^{[n\lambda^\ast]} h_1^{(n)} (\xi_i,\xi_j^{(n)})+\sum_{j=[n\lambda^\ast]+1}^{[ns]}
  \sum_{i=[n\lambda^\ast]+1}^{[ns]} h_2^{(n)}(\xi_i^{(n)},\xi_j^{(n)})
\] \[
   \qquad - \sum_{j=[n\lambda^\ast]+1}^{[ns]} \sum_{i=1}^{[n\lambda^\ast]} h  (\xi_i,\xi_j )-\sum_{j=[n\lambda^\ast]+1}^{[ns]}
  \sum_{i=[n\lambda^\ast]+1}^{[ns]} h (\xi_i ,\xi_j),
\]
while for $s\leq \lambda^\ast$, the term on the l.h.s.\ equals zero. Thus we obtain, again for $s\geq \lambda^\ast$,
\[
 \hspace{-5mm}\frac{2}{\sqrt{n}([ns]-1)}\bigg(  \sum_{1\leq i<j\leq [ns]} \big( g(X_i^{(n)},X_j^{(n)})-\theta  \big)
- \sum_{1\leq i<j \leq [ns]} \big(g(\xi_i,\xi_j)-\theta \big)\bigg) -\phi_\lambda^\ast(s)
\] \[
= \frac{2 [n\lambda^\ast] ([ns]-[n\lambda^\ast])}{n([ns]-1)} \sqrt{n}(\theta_1^{(n)}-\theta)
 +  \frac{([ns]-[n\lambda^\ast]) ([ns]-[n\lambda^\ast]-1)}{n([ns]-1)}\sqrt{n}(\theta_2^{(n)}-\theta)- \phi_\lambda^\ast(s) 
\] \[
  \quad + \frac{2([ns]-[n\lambda^\ast])}{n}\frac{1}{\sqrt{n}} \sum_{i=1}^{[n\lambda^\ast]} \big((g_1^{(n)}(\xi_i)-\theta_1^{(n)}) 
-(g_1(\xi_i)-\theta)  \big)
\] \[
 \quad +\frac{2([ns]-[n\lambda^\ast]-1)}{[ns]-1}\frac{1}{\sqrt{n}} \sum_{i=[n\lambda^\ast]+1}^{[ns]} 
 \big( (g_1(\xi_i^{(n)}-\theta_1^{(n)})-(g_1(\xi_i)-\theta)  \big)
\] \[
 \quad + \frac{2([ns]-[n\lambda^\ast]-1)}{[ns]-1}\frac{1}{\sqrt{n}}\sum_{i=[n\lambda^\ast]+1}^{[ns]} 
 \big( (g_1^{(n)}(\xi_i^{(n)})-\theta_2^{(n)})-(g_1(\xi_i)-\theta)  \big) 
\] \[
 \quad  +\frac{2}{\sqrt{n}([ns]-1)} \sum_{j=[n\lambda^\ast]+1}^{[ns]} \sum_{i=1}^{[n\lambda^\ast]} h_1^{(n)} (\xi_i,\xi_j^{(n)})+     \frac{2}{\sqrt{n}([ns]-1)}        \sum_{j=[n\lambda^\ast]+1}^{[ns]}
  \sum_{i=[n\lambda^\ast]+1}^{[ns]} h_2^{(n)}(\xi_i^{(n)},\xi_j^{(n)})
\] \[
   \quad -  \frac{2}{\sqrt{n}([nt]-1)}     \sum_{j=[n\lambda^\ast]+1}^{[ns]} \sum_{i=1}^{[n\lambda^\ast]} h  (\xi_i,\xi_j )
  -  \frac{2}{\sqrt{n}([ns]-1)}   \sum_{j=[n\lambda^\ast]+1}^{[ns]}
  \sum_{i=[n\lambda^\ast]+1}^{[ns]} h (\xi_i ,\xi_j),
\]
The right hand side converges to zero uniformly in $s\geq \lambda^\ast$. For the first four lines on the r.h.s., this follows from our assumptions. Uniform convergence of the  terms involving the degenerate kernels $h$, $h_1^{(n)}$, and $h_2^{(n)}$ follows with arguments similar to those employed in the proof of Theorem 2.5.
\end{proof}
Returning to the original notation of our paper, we define the $U$-statistic
\[
 U_k^{(n)} =\frac{1}{\binom{k}{2}} \sum_{1\leq i<j\leq k} g(X_i^{(n)},X_j^{(n)}), \; 0\leq k\leq n.
\]
Using the continuous mapping theorem, we derive the following corollary to the above theorem
\begin{corollary}
\[
\frac{1}{2\hat{\sigma}_n} \max_{1\leq k\leq n-1} \frac{k}{\sqrt{n}} |U_k^{(n)}-U_n^{(n)}|\;
\claw \; \sup_{0\leq \lambda \leq 1}
| B(\lambda) + \frac{1}{2\sigma} \left( \phi_{\lambda^\ast}(\lambda)-\lambda \phi_{\lambda^\ast}(1)\right)|.
\]
\end{corollary}


\section{Proofs of Sections \ref{sec:inf}, \ref{sec:main}, and \ref{sec:identification}}
\label{sec:app:proofs}

In this section we prove the theorems of the main text. These are Theorems \ref{theo1} (invariance principle for the sequential $U$-process), 
\ref{theo3} (consistency of the long run variance estimator), and Corollary
\ref{th:change+variance-asy} (asymptotics of the test statistic under the null), Theorem \ref{th:identification} (estimation of the change-point), and Theorem  \ref{th:localpower} (local power analysis).

\begin{proof}[Proof of Theorem \ref{theo1} (Invariance principle for the sequential $U$-process)] Using the Hoeffding decomposition, we can write the sequential $U$-process as
\begin{equation*}
\frac{[ns]}{\sqrt{n}}\left(U_{[ns]}-U\right)=\frac{2}{\sqrt{n}}\sum_{i=1}^{[ns]}g_1(\bX_i)+\frac{2}{\sqrt{n}}\frac{1}{[ns]-1}\sum_{1\leq i<j\leq [ns]}g_2(\bX_i,\bX_j).
\end{equation*}
For the first summand we find by Assumption \ref{ass3} and Lemma \ref{lem:neu1} that the sequence $(g_1(\bX_n))_{n\in\N}$ is $L_2$-NED with approximation constants $a_{k,2}=O(k^{-3})$. So we can apply Corollary 3.2 of  \citet{wooldridge:white:1988}, stating that the partial sum process $(\frac{2}{\sqrt{n}}\sum_{i=1}^{[ns]}g_1(\bX_i))_{t\in[0,1]}$ converges weakly to a Brownian motion with variance $4\sigma^2$. For the second summand, we use Lemma \ref{lem6}, which implies that $|\frac{1}{n}\sum_{1\leq i<j\leq n}g_2(\bX_i,\bX_j)|\leq Cn^{\frac{1}{4}}\log^2 (n)$ almost surely and consequently
\begin{equation*}
\sup_{s\in[0,1]}\frac{1}{[ns]\sqrt{n}}\sum_{1\leq i<j\leq [ns]}g_2(\bX_i,\bX_j)\leq Cn^{-\frac{1}{4}}\log^2 (n)\rightarrow0
\end{equation*}
almost surely as $n\rightarrow \infty$. Slutsky's theorem completes the proof.
\end{proof}

\begin{proof}[Proof of Theorem \ref{theo3} (consistency of the long run variance estimator)] We can write the variance estimator $\hat{\sigma}^2_n$ as
\[
	\hat{\sigma}_n^2 =\sum_{r=-(n-1)}^{n-1}\kappa\left(\frac{|r|}{b_n}\right)\frac{1}{n}\sum_{i=1}^{n-|r|}\hat{g}_1(\bX_i)\hat{g}_1(\bX_{i+|r|})
\ =\ \sum_{r=-(n-1)}^{n-1}\kappa\left(\frac{|r|}{b_n}\right)\frac{1}{n}\sum_{i=1}^{n-|r|}g_1(\bX_i)g_1(\bX_{i+|r|})
\]\[
\qquad \qquad +\sum_{r=-(n-1)}^{n-1}\frac{1}{n}\sum_{i=1}^{n-|r|}\left(g_1(\bX_i)g_1(\bX_{i+|r|})-\hat{g}_1(\bX_i)\hat{g}_1(\bX_{i+|r|})\right)\kappa(|k|/b_n).
\]
By Theorem 2.1 of \citet{dejong:2000}, we know that the first summand converges to $\sigma^2=\sum_{k=-\infty}^\infty \cov\left(g_1(\bX_0),g_1(\bX_k)\right)$. The second summand converges to 0 by Lemma \ref{lem11}, and the proof is complete.
\end{proof}

%
%
%
%
%
%
%
%
%
%
%
%
%
%
%
%
%

\begin{proof}[Proof of Corollary \ref{th:change+variance-asy} (asymptotic null distribution of change-point test statistic)]
Corollary \ref{th:change+variance-asy} is a special case of Corollary \ref{cor:1} for the specific kernel $g$ given by (\ref{eq:kernel}). It remains to show that, under the conditions of Corollary \ref{th:change+variance-asy}, the assumptions of Corollary \ref{cor:1} are met. Since $g$ is bounded, Assumption \ref{ass2} is satisfied for any $\delta > 0$, and it suffices that Assumption \ref{ass3} is fulfilled for some $\delta >0$. 
\end{proof}

%
%
%
%
%
%
%
%

\begin{proof}[Proof of Theorem~\ref{th:identification} (Change-point estimation)]

Under the assumptions in Theorem~\ref{th:identification} concerning the value of $\tau_{FG}$,
the function $|c(\lambda)|$, $\lambda \in [0,1]$, has a unique maximum at $\lambda = \lambda^*$, where the function $c:[0,1] \to \R$ is given by
\be \label{eq:c}
  c(\lambda) = 
  \begin{cases} 
   	\left[ 
   		(1-\lambda^{*2})\tau_F\  - (1-\lambda^*)^2 \tau_G\  - 2 \lambda^* (1 - \lambda^*) \tau_{F,G}	
   	\right]\lambda & \mbox{ for }\  0 \le \lambda < \lambda^*, 
   			\\
   	2 \lambda^* (\tau_{F,G} - \tau_G) (1-\lambda)\  +\  \lambda^{*2} (\tau_F+\tau_G- 2\tau_{F,G}) \left(\frac{1}{\lambda}-\lambda \right)& \mbox{ for  }\  \lambda^* \le \lambda \le 1.
  \end{cases}
\ee 
The proof relies on the fact that 
\be \label{eq:cp}
	\left(C_n(\lambda)\right)_{0\le\lambda\le1} = 
	\left( \frac{[\lambda n]}{n} \left( \hat{\tau}_{[\lambda n]} - \hat{\tau}_n \right) \right)_{0\le\lambda\le1}
	\ \cid \  \ 
	(c(\lambda))_{0\le\lambda\le1}
\ee
in $D[0,1]$. With the argmax theorem \citep[][Corollary 3.2.3]{vandervaart:wellner:1996} we have that 
\[
	\hat\lambda_n \ =\ 
	  \argmax_{0\le\lambda\le1} | C_n(\lambda) | \ \cip \ 
	  \argmax_{0\le\lambda\le1} | c(\lambda) | \ =\ \lambda^*.
\]
It remains to prove (\ref{eq:cp}). To simplify notation, we let $m = [\lambda^* n]$, write $\bZ_{i}^{(n)}$ short for $(X_{i}^{(n)},Y_{i}^{(n)})$ and further suppress the subscript $n$. Assume for an instant that the $\bZ_i$, $i=1,\ldots,n$, are independent. Then we have
\[
 t_n(k) = E(\hat{\tau}_k) =
 \begin{cases}
 		\ \ \tau_F & \ \mbox{ for } k \le m, \\
  	\frac{m(m-1)}{k(k-1)}\tau_F \ + \ \frac{(k-m)(k-m-1)}{k(k-1)}\tau_G \ + \ \frac{2 m (k-m)}{k(k-1)} \tau_{F,G}
  	& \ \mbox{ for } k \ge m + 1,
 \end{cases}
\]
from where we derive the mean function $c_n(\lambda) = E[C_n(\lambda)] = [\lambda n] n^{-1} (t_n([\lambda n]) - t_n(n))$, $\lambda\in [0,1]$, 
of the process of $(C_n(\lambda))_{0\le\lambda\le1}$ and observe that it converges to the function $c$.
Thus it remains to show that $\max_{0\le\lambda\le1} | C_n(\lambda)-c_n(\lambda)|$ converges to zero in probability
also under the short-range dependence assumption of Model~\ref{mod:cp}. In the following, let the $\bZ_i$, $i =1,\ldots,n$ be weakly dependent as specified by Model~\ref{mod:cp}.   
We will prove that 
\be \label{eq:wendler:cp}
	\max_{m <k\leq n}\frac{k}{n}\left|\hat{\tau}_{k}-t_n(k)\right| \cip 0.
\ee  
The convergence of 
$\max_{1<k\leq m}\frac{k}{n}\left|\hat{\tau}_{k}-t_n(k)\right|$ follows along the same lines. Hence the maximum in (\ref{eq:wendler:cp}) can be extended to the range $k = 1, \ldots, n$.
%
We split the difference $\frac{k}{n}\left|\hat{\tau}_{k}-t_n(k)\right|$ into three parts: two one-sample $U$-statistics and one two-sample $U$-statistic with the kernel
\begin{equation*}
 g\left((x_1,y_1),(x_2,y_2)\right)=\bEins_{(0,\infty)}((x_2-x_1)(y_2-y_1))-\bEins_{(-\infty,0)}((x_2-x_1)(y_2-y_1))
\end{equation*}
as in Section \ref{sec:main}. By the triangle inequality we get:
\[
	\frac{k}{n}\left|\hat{\tau}_{k}-t_n(k)\right| 
	\ \leq\  
	\left|\frac{2}{n(k-1)}\sum_{1\leq i<j \leq m}\left(g(\bZ_i,\bZ_j)-\tau_F\right)\right|
\]
\be \label{eq:triangle}	
\qquad + \ \left|\frac{2}{n(k-1)}\sum_{m+1\leq i<j \leq k}\left(g(\bZ_i,\bZ_j)-\tau_G\right)\right|
 \  + \ \left|\frac{2}{n(k-1)}\sum_{1\leq i\leq m<j \leq k}\left(g(\bZ_i,\bZ_j)-\tau_{F,G}\right)\right|.
\ee

\smallskip
\noindent
For the first summand on the right-hand side, we have by Theorem \ref{theo1}:
\[ 
\textstyle
\max\limits_{m<k\leq n}\left|\frac{2}{n(k-1)}\sum\limits_{1\leq i<j \leq m}\left(g(\bZ_i,\bZ_j)-\tau_F\right)\right|
\leq \left|\frac{2}{n(m-1)}\sum\limits_{1\leq i<j \leq m}\left(g(\bZ_i,\bZ_j)-\tau_F\right)\right| \cip 0.
\] 
Due to our assumptions, $(\bZ_{i}^{(n)})_{m+1\leq i \leq n}$ is a stationary process which satisfies Condition \ref{ass3}, so we also treat the second summand by Theorem \ref{theo1}.
For the third summand, we apply a two-sample Hoeffding decomposition 
 $ g(\bz_1,\bz_2)=\tau_{F,G}+\tilde{g}_1(\bz_1)+\tilde{g}_2(\bz_2)+\tilde{g}_3(\bz_1,\bz_2)$
with
\[ 
	\tilde{g}_1(\bz_1) =E\left(g(\bz_1,\bZ_{m+1})\right)-\tau_{F,G}, \qquad
	\tilde{g}_2(\bz_1) =E\left(g(\bZ_1,\bz_2)\right)-\tau_{F,G},
\]\[
	\tilde{g}_3(\bz_1,\bz_2) =g(\bz_1,\bz_2)-\tilde{g}_1(\bz_1)-\tilde{g}_2(\bz_2)-\tau_{F,G},
\]
where $\bz_1, \bz_2 \in \R^2$. We get
\[ 
 \frac{2}{n(k-1)}\sum\limits_{1\leq i\leq m<j \leq k} \! \! \left(g(\bZ_i,\bZ_j)-\tau_{F,G}\right) 
\]\be \label{eq:3summands} 
 = \, \frac{2(k-m)}{n(k-1)}\!\sum\limits_{i=1}^{m} \! \tilde{g}_1(\bZ_i) 
 +   \frac{2 m}{n(k-1)} \! \sum\limits_{j=m+1}^{k} \!\tilde{g}_2(\bZ_j) 
 +  \frac{2}{n(k-1)} \sum\limits_{1\leq i\leq m<j \leq k}\tilde{g}_3(\bZ_i,\bZ_j). 
\ee
To obtain a maximal inequality, we use Theorem 2.4.1 of \citet{stout:1974}: For random variables $R_1,\ldots,R_n$ with $E(\sum_{j=k}^{k+l-1}R_j)^2\leq C_1l$ for a constant $C_1$, we have that
\be \label{eq:stout} 
	E\left[ \max_{1\leq l\leq n}\Big( \sum\nolimits_{j=1}^lR_j\Big)^2 \right] 
	\leq C_1n\left(\log(2n)/\log2 \right)^2.
\ee 
We define the random variables $R_j=\sum_{i=1}^{m}\tilde{g}_3(\bZ_i,\bZ_{j+m})$. Without loss of generality, we can assume that the random variables $\bZ_i$ are bounded and thus the process is $L_1$-NED by Lemma \ref{lem:l_0}(\ref{lem:l_0(ii)}). Furthermore, the kernel $\tilde{g}_3$ is degenerate, so we can apply Proposition A.2 of \citet{dehling:fried:2012} to obtain the moment bound
$	E\big(\sum_{j=1}^{k+l-1}R_j\big)^2
	= 
	E\big[
		\sum_{j=1}^{k+l-1} \sum_{i=1}^{m}\tilde{g}_3(\bZ_i,\bZ_{j+m})
	\big]^2\leq C m l$.
Applying (\ref{eq:stout}), we find
\[
	E\left(\max_{m<k\leq n}
	    \left|\frac{2}{n(k-1)}\sum\limits_{1\leq i\leq m<j  \leq k }\tilde{g}_3(\bZ_i,\bZ_j)\right|
	   \right)^2 \ \leq  
		m^{-2} E\left(\max\limits_{1<k\leq n-m}\left|\frac{2}{n}\sum_{j=1}^{k}R_j\right|\right)^2  \leq 
		\frac{C \log^2(2n)}{m n\log^2 2}
\]
converges to zero as $n\rightarrow0$. Thus the third summand in (\ref{eq:3summands}) converges to zero in probability. 
As for the first two summands, we have that $E\big[\sum_{j=k}^{k+l-1}\tilde{g}_2(\bZ_j)\big]^2 \leq Cl$, since
	$l^{-1} \Var\left[ \sum_{j=k}^{k+l-1}\tilde{g}_2(\bZ_j) \right]$
converges to a finite limit as $l \to \infty$. Hence, (\ref{eq:stout}) applied to $R_j=\tilde{g}_2(\bZ_{j+m})$ leads to
\[
	\max_{m<k\leq n}\left|\frac{2}{n}\sum\nolimits_{j=m+1}^{k}\tilde{g}_2(\bZ_j)\right| \cip 0
\] 
as $n \to \infty$. Finally, $2/n\sum_{i=1}^{m}\tilde{g}_1(\bZ_i)\cip 0$, as its variance converges to zero. We have thus shown (\ref{eq:wendler:cp}), which completes the proof.
\end{proof}

\begin{proof}[Proof of Theorem \ref{th:localpower}]

It remains to calculate the quantities $c_1,c_2$ introduced above for the change-point test based on Kendall's $\tau$, under the local alternative as specified in Model~\ref{mod:la}.  Observe that, in this case, $\theta =\tau_F$, $\theta_1^{(n)} = \tau_{F,F_n}$, and 
$\theta_2^{(n)} = \tau_{F_n,F_n}$, 
where $F$ denotes the joint distribution of $(X_i,Y_i)$ before the change, and $F_n$ denotes the joint distribution of $(X_i^{(n)},Y_i^{(n)})$ after the change. 

We now provide a formula for $\tau_{F,G}$ in the case when both distributions $F, G$ are absolutely continuous, with densities $f(x,y)$ and $g(x,y)$.   Let $(X_1,Y_1)$ and $(X_2,Y_2)$ be random variables with densities $f(x,y)$, and $g(x,y)$, respectively, and let $(X_1,Y_1)$ be independent of $(X_2,Y_2)$. Then we have
\begin{equation}
\tau_{F,G}=2(\tau_{F,G}^{(I)}+\tau_{F,G}^{(II)})-1,
\end{equation}
where 
\begin{eqnarray*}
 \tau_{F,G}^{(I)} &=& P(X_1\leq X_2, Y_1 \leq Y_2) 
  =\int_{-\infty}^\infty \int_{x_1}^\infty \int_{-\infty}^\infty \int_{y_1}^\infty f(x_1,y_1) g(x_2,y_2) dy_2 dy_1 dx_2 dx_1 \\
 \tau_{F,G}^{(II)}
 &=& P(X_1\geq X_2, Y_1 \geq Y_2) =
 \int_{-\infty}^\infty \int_{x_2}^\infty \int_{-\infty}^\infty \int_{y_2}^\infty f(x_1,y_1) g(x_2,y_2) dy_2 dy_1 dx_2 dx_1. 
\end{eqnarray*}
Noting that $\tau_F=\tau_{F,F}$, we obtain $\tau_F=2(\tau_{F,F}^{(I)}+\tau_{F,F}^{(II)})-1$, where $\tau_{F,F}^{(I)}$ and $\tau_{F,F}^{(II)}$ are given by the above integrals with $f=g$. Finally, we obtain
\[
  \tau_{F,G}-\tau_F = 2\left((\tau_{F,G}^{(I)}-\tau_{F,F}^{(I)})+(\tau_{F,G}^{(II)} -\tau_{F,F}^{(II)})\right)
\]

Under Model~\ref{mod:la}, we have $G=F_n$, where $F_n$ is the distribution of 
$(X_1+\frac{\Delta}{\sqrt{n}}Y_1,Y_1)$. By the transformation formula for densities, $F_n$ has the density
\[
   f_n(x_2,y_2)=f(x_2-\frac{\Delta}{\sqrt{n}} y_2,y_2).
\]
Thus, we obtain
\[
  \tau_{F,F_n}^{(I)}-\tau_{F,F}^{(I)}=\int_{-\infty}^\infty \int_{x_1}^\infty \int_{-\infty}^\infty \int_{y_1}^\infty f(x_1,y_1) \left(f(x_2-\frac{\Delta}{\sqrt{n}}y_2,y_2) -f(x_2,y_2)\right) dy_2 dy_1 dx_2 dx_1.
\]
Under the assumptions made in Model \ref{mod:la} on the densities, we then obtain 
\[
  \lim_{n\rightarrow \infty} \sqrt{n}(\tau_{F,F_n}^{(I)}-\tau_{F,F}^{(I)})=- \Delta \int_{-\infty}^\infty \int_{x_1}^\infty \int_{-\infty}^\infty \int_{y_1}^\infty f(x_1,y_1) y_2 f_1(x_2,y_2) dy_2 dy_1 dx_2 dx_1,
\]
where $f_1(x,y)=\frac{\partial}{\partial x}f(x,y)$. Furthermore, 
\[
 \int_{x_1}^\infty f_1(x_2,y_2) dx_2= -f(x_1,y_2),
\]
and thus 
\[
  \lim_{n\rightarrow \infty} \sqrt{n}(\tau_{F,F_n}^{(I)}-\tau_{F,F}^{(I)}) \ = \ \Delta \int_{-\infty}^\infty  \int_{-\infty}^\infty \int_{y_1}^\infty f(x_1,y_1) y_2 f(x_1,y_2) dy_2 dy_1 dx_1
\] 
\[ = \ \Delta \int_{-\infty}^\infty \left(\int_{-\infty}^\infty \int_{-\infty}^{y_2} y_2 f(y_1|x) f(y_2|x) dy_1 dy_2    \right) f_X^2(x) dx 
\]\[
   = \ \Delta \int_{-\infty}^\infty \left(\int_{-\infty}^\infty y_2 F(y_2|x) f(y_2|x) dy_2    \right) f_X^2(x) dx 
	\ = \   \Delta \int_{-\infty}^\infty \left(\int_{-\infty}^\infty y_2 F(y_2|x) F(dy_2|x)    \right) f_X^2(x) dx,
\]
where $f_X$ denotes the marginal density of $X$, and where $f(y|x)$ and $F(y|x)$ denote the conditional density, respectively the conditional distribution function of $Y$ given $X=x$.

With similar calculations, we obtain
\[
 \lim_{n\rightarrow \infty} \sqrt{n}(\tau_{F,F_n}^{(II)}-\tau_{F,F}^{(II)})=
 \Delta \int_{-\infty}^\infty \left(\int_{-\infty}^\infty y_2  (F(y_2|x) -1) F(dy_2|x)    \right) f_X^2(x) dx.
\]
Finally, one can show that $\lim_{n\rightarrow \infty} \sqrt{n} (\tau_{F_n,F_n}-\tau_{F,F}) = 2 \lim_{n\rightarrow \infty} \sqrt{n} (\tau_{F,F_n}-\tau_{F,F})$.
\end{proof}

\end{document}


%
%
%
%

\maketitle
%
%


\bigskip
\begin{center}
{\bf Abstract}
\end{center}

This document contains proofs and further technical results for the article ``Testing for changes in Kendall's tau''.
It consists of one section, labeled Appendix D. The results are labeled D.1, D.2, $\ldots$. Results from the main document are referred to as in the main document. The labels contain references to the respective section, e.g., Corollary 3.1 can be found in Section 3, Lemma A.2 in Appendix A of the main document.

\appendix
\setcounter{section}{3}
\section{Proofs of lemmas}



We first prove Lemmas \ref{lem:l_0} and \ref{lem:neu1} from Appendix \ref{sec:app:pned}.

\begin{proof}[Proof of Lemma \ref{lem:l_0}]
Part (\ref{lem:l_0(i)}) is straightforward. \par
\emph{Part (\ref{lem:l_0(ii)}):} There are positive constants $C_1, C_2$ such that
\begin{eqnarray}
a^p_{p,k} 			\ = \ 			E\left| \bX_0 - E\left( \bX_0 \middle| \F_{-k}^k \right) \right|_p^p 
\nonumber		& \le & C_1 E\left| \bX_0 - E\left( \bX_0 \middle| \F_{-k}^k \right) \right|_1 \\
\label{eq:2}		& \le & C_2 E\left| \bX_0 - f_{k}\left(\bZ_{-k},\ldots,\bZ_k \right) \right|_1.
\end{eqnarray}
The first inequality is due to the boundedness of $\bX_0$. The second inequality can be shown by applying Jensen's inequality for the conditional expectation $E(\cdot | \F_{-k}^k  )$ to the convex function $|\cdot|_1$. Further, for any $\varepsilon > 0$ we have
\[
	E\left| \bX_0 - f_{k}\left(\bZ_{-k},\ldots,\bZ_k \right) \right|_1
	\ \le \ 
	C_3 P\left( \left| \bX_0 - f_{k}\left(\bZ_{-k},\ldots,\bZ_k \right) \right|_1 > \varepsilon \right)
 +   \varepsilon
	\ \le \  
	C_3 \Phi(\varepsilon) a_k
	 +   \varepsilon
\]
Combining this with (\ref{eq:2}) we arrive at $a^p_{p,k}  \le  C_2 C_3 \Phi(\varepsilon) a_k  +   C_2 \varepsilon$.
By first choosing $\varepsilon$ sufficiently small and then $k$ sufficiently large we can make the left-hand side arbitrarily small, and hence $(\bX_n)_{n \in \Z}$ is $L_p$-NED, $p \ge 1$, on $(\bZ_n)_{n \in \Z}$.
In particular, if condition (\ref{eq:ned.con}) holds, we get 
$a^p_{p,k}  \le  C_2 C_3 \Phi(s_k) a_k  +   C_2 s_k  =  O(s_k)$.
%
%
%
\emph{Part (\ref{lem:l_0(iii)}):}
Let $f_k(\bZ_{-k},\ldots,\bZ_k) = E(\bX_0|\F_{-k}^k)$. By means of the H\"older and the Markov inequalities we have for every $\varepsilon > 0$:
\[
	P\left( \left| \bX_0 - f_{k}\left(\bZ_{-k},\ldots,\bZ_k \right) \right|_1 > \varepsilon \right)
	\, \le \,
  	q^{p-1} \varepsilon^{-p }\, E\left| \bX_0 - E(\bX_0|\F_{-k}^k) \right|_p^p 
\]
Choosing $\Phi(\varepsilon) = q^{p-1} \varepsilon^{-p}$ and $a_k = a_{p,k}^p$ we have $a_k \to 0$ as $k \to \infty$, and $(\bX_n)_{n \in \Z}$ is hence $P$-NED on $(\bZ_n)_{n \in \Z}$.
\end{proof}



\begin{proof}[Proof of Lemma \ref{lem:neu1}] By the definition of $g_1$, we have that, for any $\bx,\bx'\in\R^d$,
\begin{equation*}
\left|g_1(\bx)-g_1(\bx')\right|=\left|Eg(\bx,\bX_0)-Eg(\bx',\bX_0)\right|\leq \int\left|g(\bx,\by)-g(\bx',\by)\right|dF(\by)
\end{equation*}
and consequently, for independent copies $\bX,\bY$ of $\bX_0$,
\begin{multline*}
E\left(\sup_{|\bx-\bX|\leq\epsilon}\left|g_1(\bx)-g_1(\bX)\right|\right)^2 \\
\leq E\left(\sup_{|\bx-\bX|\leq\epsilon}\int\left|g(\bx,\by)-g(\bX,\by)\right|dF(\by)\right)^2
\leq E\left(\int\sup_{|\bx-\bX|\leq\epsilon}\left|g(\bx,\by)-g(\bX,\by)\right|dF(\by)\right)^2 \\
\leq E\left(\sup_{|\bx-\bX|\leq\epsilon}\left|g(\bx,\bY)-g(\bX,\bY)\right|\right)^2
\leq E\left(\sup_{|\bx-\bX|\leq\epsilon,|\by-\bY|\leq\epsilon}\left|g(\bx,\by)-g(\bX,\bY)\right|\right)^2\leq L\epsilon.
\end{multline*}
Recall that the conditional expectation minimizes the $L_2$-distance, so
\begin{equation*}
	E \left\{ g_1(\bX_0) - E (g_1(\bX_0) | \F^k_{-k}) \right\}^2 
	\leq 
	E \left\{ g_1(\bX_0) - g_1(\bX_{0,k}) \right\}^2,
\end{equation*}
where $\bX_{0,k} = f_k(\bZ_{-k},\ldots, \bZ_k)$. 
Now we will make use of the $P$-near epoch dependence and the H\"older inequality and obtain
\begin{multline*}
E \left( g_1(\bX_0) - g_1(\bX_{0,k}) \right)^2\\
=  E \left( g_1(\bX_0) - g_1(\bX_{0,k}) \right)^2 \bEins_{\{|\bX_0-\bX_{0,k}|>s_l\}}
	+ E\left( g_1(\bX_0) - g_1(\bX_{0,k}) \right)^2 \bEins_{\{|\bX_0-\bX_{0,k}|\leq s_l\} }\\
\leq 
\left(2\left\|g_1^2(\bX_0)\right\|_{\frac{2+\delta}{2}}+2\left\|g_1^2(\bX_{0,k})\right\|_{\frac{2+\delta}{2}}\right)
		\left(P(|\bX_0-\bX_{0,k}|>s_l)\right)^{\frac{\delta}{2+\delta}} \\
+ E\left( g_1(\bX_0) - g_1(\bX_{0,k}) \right)^2 \bEins_{\{|\bX_0-\bX_{0,k}|\leq s_l\}}
\leq C(s_l^{\frac{2+\delta}{\delta}})^{\frac{\delta}{2+\delta}}+ L s_l
\leq Cs_l
\end{multline*}
and finally
\begin{equation*}
\left\|g_1(\bX_0) - E (g_1(\bX_0) | \F^k_{- k})\right\|_2\leq  \left(E \left( g_1(\bX_0) - g_1(\bX_{0,k}) \right)^2\right)^{1/2}\leq Cs_l^{1/2},
\end{equation*}
which completes the proof.
\end{proof}
%



The following lemmas \ref{lem2} to \ref{lem6} are required for the proof of Theorem \ref{theo1} (Invariance principle for the sequential $U$-process).

\begin{lemma}\label{lem2} Let $(\bZ_n)_{n\in\Z}$ be a stationary and absolutely regular stochastic process with mixing coefficients $(\beta_k)_{k\in\N}$. Then there exist processes $(\bZ'_n)_{n\in\Z}$ and $(\bZ''_n)_{n\in\Z}$, independent of each other and both with the same distribution as $(\bZ_n)_{n\in\Z}$ such that
$P\{(\bZ'_n)_{n\geq k}=(\bZ_n)_{n\geq k}\}=1-\beta_k$ and $P\{(\bZ''_n)_{n\leq 0}=(\bZ_n)_{n\leq 0}\}=1-\beta_k$.
\end{lemma}
This can be proved in a similar way to Lemma 2.5 of \citet{borovkova:burton:dehling:2001}.

\begin{lemma}\label{lem3} Let $(\bX_n)_{n\in\Z}$ be $P$-NED of an absolutely regular sequence $(\bZ_n)_{n\in\Z}$ and $g$ be a kernel, such that Assumptions \ref{ass1} and \ref{ass2} hold. Then there is a constant $C>0$, such that for $i,k,l\in\N$, $\epsilon>0$:
\begin{equation*}
	\left\|g_2(\bX_i,\bX_{i+k+2l})-g_2(\bX_{i,l},\bX_{i+k+2l,l})\right\|_2
	\leq 
	C (\sqrt{\epsilon}+a_l^{\frac{\delta}{2+\delta}}\Phi^{\frac{\delta}{2+\delta}}(\epsilon)+\beta^{\frac{\delta}{2+\delta}}_k), 
\end{equation*}
where $\bX_{i,l}=f_l(\bZ_{i-l},\ldots,\bZ_{i+l})$.
\end{lemma}
\begin{proof}[Proof of Lemma \ref{lem3}] First note that we can rewrite any $\bX_i$ such that
\begin{equation*}
	\bX_i	=	f_\infty\left((\bZ_{i+n})_{n\in\Z}\right).
\end{equation*}
By Lemma \ref{lem2} there exist independent copies $(\bZ'_n)_{n\in\Z}$ and $(\bZ''_n)_{n\in\Z}$ of $(\bZ_n)_{n\in\Z}$ satisfying
$P\{(\bZ'_n)_{n\geq i+l+k}=(Z_n)_{n\geq i+l+k}\}=1-\beta_k$ and $P\{(\bZ''_n)_{n\leq i+l}=(Z_n)_{n\leq i+l}\}=1-\beta_k$.
We define $\bX'_i=f_{\infty}((\bZ'_{n+i})_{n\in\Z})$ and $\bX''_i=f_{\infty}((\bZ''_{n+i})_{n\in\Z})$.
Analogously, $\bX'_{i,l}=f_{l}(\bZ'_{i-l},\ldots,\bZ'_{i+l})$ and $\bX''_{i,l}=f_{l}(\bZ'_{i-l},\ldots,\bZ'_{i+l})$.  Now we can conclude that
\begin{multline*}
	\left\|g_2(\bX_i,\bX_{i+k+2l})-g_2(\bX_{i,l},\bX_{i+k+2l,l})\right\|_2 
	\leq \left\|g_2(\bX_i,\bX_{i+k+2l})-g_2(\bX''_{i},\bX'_{i+k+2l})\right\|_2 \\
	+	\left\|g_2(\bX''_i,\bX'_{i+k+2l})-g_2(\bX''_{i,l},\bX'_{i+k+2l,l})\right\|_2 
	+ \left\|g_2(\bX''_{i,l},\bX'_{i+k+2l,l})-g_2(\bX_{i,l},\bX_{i+k+2l,l})\right\|_2 
\end{multline*}
We will treat the three summands on the right-hand side separately. For the first summand, we use the H\"older inequality to obtain
\begin{multline*}
	\left\|g_2(\bX_i,\bX_{i+k+2l})-g_2(\bX''_{i},\bX'_{i+k+2l})\right\|_2\\
	= \big\| \{g_2(\bX_i,\bX_{i+k+2l})-g_2(\bX''_{i},\bX'_{i+k+2l})\}
	\bEins_{\left\{(\bZ'_n)_{n\geq i+l+k}\neq(\bZ_n)_{n\geq i+l+k}\text{ or } (\bZ''_n)_{n\leq i+l}\neq(\bZ_n)_{n\leq i+l}\right\}} \big\|_2 \\
	\leq 	
	\left\|(g_2(\bX_i,\bX_{i+k+2l})-g_2(\bX''_{i},\bX'_{i+k+2l})\right\|_{\frac{2+\delta}{2}}\\ 
\times \left(P\{(\bZ'_n)_{n\geq i+l+k}\neq(\bZ_n)_{n\geq i+l+k}\text{ or } (\bZ''_n)_{n\leq i+l}\neq(\bZ_n)_{n\leq i+l}\}\right)^{\frac{\delta}{2+\delta}}
	\leq 2M^{2/(2+\delta)}2\beta^{\delta/(2+\delta)}_k
\end{multline*}
with Assumption \ref{ass2} and Lemma \ref{lem2}. With the same arguments, the third summand is also bounded by $C\beta^{\delta/(2+\delta)}_k$. For the second summand, we make use of the variation condition and the $P$-NED property
\begin{multline*}
	\left\|g_2(\bX''_i,\bX'_{i+k+2l})-g_2(\bX''_{i,l},\bX'_{i+k+2l,l})\right\|_2\\
	\leq 
	\big\|\{ g_2(\bX''_i,\bX'_{i+k+2l})-g_2(\bX''_{i,l},\bX'_{i+k+2l,l})\}
			\bEins_{\left\{\left|\bX''_i-\bX''_{i,l}\right|\leq\epsilon, \left|\bX'_{i+k+2l}-\bX'_{i+k+2l,l}\right|\leq\epsilon\right\}}\big\|_2 \\
	+ \big\|\{(g_2(\bX''_i,\bX'_{i+k+2l})-g_2(\bX''_{i,l},\bX'_{i+k+2l,l})\}
			\bEins_{\left\{\left|\bX''_i-\bX''_{i,l}\right|>\epsilon\text{ or } \left|\bX'_{i+k+2l}-\bX'_{i+k+2l,l}\right|>\epsilon\right\}}\big\|_2 \\
\leq \sqrt{L\epsilon}+2M^{2/(2+\delta)}2P\left(\left|\bX_{i+k+2l}-f_l(\bZ'_{i+k+l},\ldots,\bZ'_{i+k+3l})\right|>\epsilon\right)^{\frac{\delta}{2+\delta}} 
\leq C( \sqrt{\epsilon} +a_l^{\frac{\delta}{2+\delta}}\Phi^{\frac{\delta}{2+\delta}}(\epsilon)).
\end{multline*}
The proof is complete.
\end{proof}

\begin{lemma}\label{lem4} Under Assumptions \ref{ass3}, \ref{ass2} and \ref{ass1}, we have for $n_1<n_2\leq n$:
\begin{equation*}
\bigg\|\sum_{\substack{1\leq i<j\\n_1<j\leq n_2}}\big(g_2(\bX_i,\bX_j)-g_2(\bX_{i,l},\bX_{j,l})\big)\bigg\|_2\leq C(n_2-n_1)n^{1/4}.
\end{equation*}
\end{lemma}

\begin{proof}[Proof of Lemma \ref{lem4}] We set $l=\lfloor n^{1/4}\rfloor$ and $\epsilon=l^{-6}$. If $k<0$, we set $\beta_k=1.$ With Lemma \ref{lem3} and some straightforward calculus, we obtain
\begin{multline*}
	\left\|\sum_{\substack{1\leq i<j\\n_1<j\leq n_2}}\left(g_2(\bX_i,\bX_j)-g_2(\bX_{i,l},\bX_{j,l})\right)\right\|_2 
	\leq 
	C(n_2-n_1) \sum_{k=1}^n (\sqrt{\epsilon}+a_l^{\frac{\delta}{2+\delta}}\Phi^{\frac{\delta}{2+\delta}}(\epsilon)+\beta^{\frac{\delta}{2+\delta}}_{k-l})   \\
	\leq 
	C(n_2-n_1)\sum_{k=1}^n 
	\left( \lfloor n^{1/4}\rfloor^{-3} + a_l^{\frac{\delta}{2+\delta}}\Phi^{\frac{\delta}{2+\delta}}(\lfloor n^{1/4}\rfloor^{-6})+\beta^{\frac{\delta}{2+\delta}}_{k-\lfloor n^{1/4}\rfloor} \right)
\leq C(n_2-n_1)n^{1/4}. 
\end{multline*}
\end{proof}

As we approximate the random variables $\bX_i$ by $\bX_{i,l}=f_l(\bZ_{i-l},\ldots,\bZ_{i+l})$, we introduce the Hoeffding decomposition of the kernel 
$g$ with respect to these approximating random variables. Let $(\tilde{\bZ}_n)_{n\in\Z}$ be an independent copy of the sequence $(\bZ_n)_{n\in\Z}$ and $\tilde{\bX}_{i,l}=f_l(\tilde{\bZ}_{i-l},\ldots,\tilde{\bZ}_{i+l})$. We define
\[
U_l = Eg(\bX_{0,l},\tilde{\bX}_{0,l}), \quad 
g_{1,l}(x) = Eg(\bx,\tilde{\bX}_{0,l})-U_l, \quad 
g_{2,l}(\bx,\by) = g(\bx,\by)-g_{1,l}(\bx)-g_{1,l}(\by)-U_l.
\]

\begin{lemma}\label{lem4a} Under Assumptions \ref{ass3}, \ref{ass2} and \ref{ass1}, we have for $n_1<n_2\leq n$:
\begin{equation*}
	\bigg\|\sum_{\substack{1\leq i<j\\n_1<j\leq n_2}}\left(g_{2,l}(\bX_{i,l},\bX_{j,l})-g_2(\bX_{i,l},\bX_{j,l})\right)\bigg\|_2
	\leq 
	C(n_2-n_1)n^{1/4}.
\end{equation*}
\end{lemma}

\begin{proof}[Proof of Lemma \ref{lem4a}] Let $(\tilde{\bZ}_n)_{n\in\Z}$ be an independent copy of the sequence $(\bZ_n)_{n\in\Z}$ and $\tilde{\bX}_{i}=f_\infty((\tilde{\bZ}_{n+l})_{l\in\Z})$. Then
$g_2(\bx,\by) = g(\bx,\by)-Eg(\bx,\tilde\bX_j)-Eg(\tilde\bX_i,\by)+Eg(\tilde{\bX}_i,\bX_j)$ and
$g_{2,l}(\bx,\by)= g(\bx,\by)-Eg(\bx,\tilde{\bX}_{j,l})-Eg(\tilde{\bX}_{i,l},\by)+Eg(\tilde{\bX}_{i,l},\bX_{j,l})$.
for every $i, j, l \in \N$.
So we can conclude that
\[ 
	\left\|g_{2,l}(\bX_{i,l},\bX_{j,l})-g_2(\bX_{i,l},\bX_{j,l})\right\|_2
	\leq \left\|g(\bX_{i,l},\tilde{\bX}_{j,l})-g(\bX_{i,l},\tilde{\bX}_{j})\right\|_2
\]\[	
 \qquad	+ \left\|g(\tilde{\bX}_{i,l},\bX_{j,l})-g(\tilde{\bX}_{i},\bX_{j,l})\right\|_2
	+ \left\|g(\tilde{\bX}_{i,l},\bX_{j,l})- g(\tilde{\bX}_{i},\bX_j)\right\|_2 
\]
\[
	\leq C(\sqrt{\epsilon}+a_l^{\frac{\delta}{2+\delta}}\Phi^{\frac{\delta}{2+\delta}}(\epsilon)+\beta^{\frac{\delta}{2+\delta}}_k),
\]
where the bound in the last line can be proved along the lines of Lemma \ref{lem3}. The assertion of Lemma \ref{lem4a} then follows analogously to Lemma \ref{lem4}.
\end{proof}

\begin{lemma}\label{lem5}
Under Assumptions \ref{ass3}, \ref{ass2} and \ref{ass1}, we have for $n_1 < n_2\leq n$
\begin{equation*}
	E\left(\sum_{\substack{1\leq i<j\\n_1<j\leq n_2}}g_{2,l}(\bX_{i,l},\bX_{j,l})\right)^2\leq C(n_2-n_1)\,n\,l^2.
\end{equation*}
\end{lemma}
\begin{proof}[Proof of Lemma \ref{lem5}] By Lemma 1 of \citet{yoshihara:1976} we obtain
\begin{equation*}
		\left|E \{ g_{2,l}(\bX_{i_{(1)},l},\bX_{i_{(2)},l})g_{2,l}(\bX_{i_{(3)},l},\bX_{i_{(4)},l}) \} \right|\leq C\beta^{\delta/(2+\delta)}_{m-l}
\end{equation*}
with $m=\max\left\{i_{(2)}-i_{(1)},i_{(4)}-i_{(3)}\right\}$, where $i_{(1)}\leq i_{(2)} \leq i_{(3)}\leq i_{(4)}$ are the ordered indices $i_1,i_2,i_3,i_4$. 
To simplify the notation, we define $\beta_{m-l}=1$ for $m-l<0$. Note that, for given $m$, we have less than $n$ choices for $i_{(1)}$ and less than $n_2-n_1$ choices for $i_{(4)}$, which has to be either $i_2$ or $i_4$. If $m = i_{(2)}-i_{(1)}$, there are $m$ possibilities for $i_{(3)}$, and if $m=i_{(4)}-i_{(3)}$, there are $m$ possibilities for $i_{(2)}$. We conclude that
\begin{multline*}
	E\Bigg(\sum_{\substack{1\leq i<j\\n_1<j\leq n_2}}g_{2,l}(\bX_{i,l},\bX_{j,l})\Bigg)^2
	=	\sum_{\substack{1\leq i_1<j_1\\n_1<j\leq n_2}}\sum_{\substack{1\leq i<j\\n_1<j\leq n_2}}C\beta^{\delta/(2+\delta)}_{m-l}
	\ \leq \ C(n_2-n_1)n \sum_{m=1}^n m\beta^{\delta/(2+\delta)}_{m-l} \\
	\leq C(n_2-n_1)n\left(\sum_{m=1}^l m+l^2\sum_{m=l+1}^n (m-l) \beta^{\delta/(2+\delta)}_{m-l}\right)\leq C(n_2-n_1)nl^2.
\end{multline*}
The proof is complete.
\end{proof}

\begin{lemma}\label{lem6} Under Assumptions \ref{ass3}, \ref{ass2} and \ref{ass1}, we have
\begin{enumerate}[(i)]
\item
$ \displaystyle \Big\|\max_{n\leq 2^k}\sum_{1\leq i<j\leq n}g_2(\bX_i,\bX_j)\Big\|_2\leq C 2^{\frac{5}{4}k}k $ \ and
\item
$ \sum_{1\leq i<j\leq n}g_2(\bX_i,\bX_j)=O\left(n^{\frac{5}{4}}\log^2 (n)\right)$ almost surely.
\end{enumerate}
\end{lemma}
\begin{proof}[Proof of Lemma \ref{lem6} (i):] We will use Lemma \ref{lem4} and Lemma \ref{lem5} with $l=l_k=\lfloor 2^{\frac{1}{4}k}\rfloor$ and split the expectation into three parts:
\begin{multline*}
	\bigg\|\max_{n\leq 2^k}\Big|\sum_{1\leq i<j\leq n}g_2(\bX_i,\bX_j)\Big|\bigg\|_2 
	\leq \bigg\|\max_{n\leq 2^k}\Big|\sum_{1\leq i<j\leq n}\left(g_2(\bX_i,\bX_j)-g_2(\bX_{i,l},\bX_{j,l})\right)\Big|\bigg\|_2 \\
	+\bigg\|\max_{n\leq 2^k}\Big|\sum_{1\leq i<j\leq n}(g_{2,l}(\bX_{i,l},\bX_{j,l})-g_2(\bX_{i,l},\bX_{j,l}))\Big| \bigg\|_2\\
	+\bigg\|\max_{n\leq 2^k}\Big|\sum_{\substack{1\leq i<j\leq n}}g_{2,l}(\bX_{i,l},\bX_{j,l})\Big|\bigg\|_2=I_k+I\!\! I_k+I\!\! I\!\! I_k.
\end{multline*}
Now by Lemma \ref{lem4}
\begin{multline*}
	I_k  
	\leq \bigg\|\sum_{j=1}^{2^k}\Big|\sum_{1\leq i<j}\left(g_2(\bX_i,\bX_j)-g_2(\bX_{i,l},\bX_{j,l})\right)\Big|\bigg\|_2 \\
	\leq \sum_{j=1}^{2^k}\bigg\|\sum_{1\leq i<j}\left(g_2(\bX_i,\bX_j)-g_2(\bX_{i,l},\bX_{j,l})\right)\bigg\|_2
		\leq C 2^{\frac{5}{4}k}.
\end{multline*}
Similarly, we get by Lemma \ref{lem4a} that $I\!\! I_k\leq C 2^{\frac{5}{4}}$. To deal with $I\!\! I\!\! I_k$, we define the random variables
$Y_{j,l}=\sum_{1\leq i<j}g_{2,l}(\bX_{i,l},\bX_{j,l})$
and rewrite $I\!\! I\!\! I_k$ as $\big\|\max_{n\leq 2^k}|\sum_{j=1}^n Y_{j,l}|\big\|_2$.
As we have $E(\sum_{j=n_1}^{n_2} Y_{j,l})^2\leq C(n_2-n_1)n^{3/2}$ by Lemma \ref{lem5}, we can use Theorem 1 of \citet{moricz1976moment} to obtain
$ E(\max_{n\leq 2^k}\sum_{j=1}^n Y_{j,l})^2\leq C2^{5k/2}k^2$, which completes the proof of part (1). 

\emph{Part (ii):} It suffices to prove that
\begin{equation*}
\max_{n\leq 2^k}\sum_{1\leq i<j\leq n}g_2(\bX_i,\bX_j)=O\left( 2^{\frac{5}{4}k}k^2\right)
\end{equation*}
almost surely. By means of the Chebychev inequality, we have for any $\epsilon>0$
\begin{multline*}
\sum_{k=1}^\infty P\left(\max_{n\leq 2^k}\sum_{1\leq i<j\leq n}g_2(\bX_i,\bX_j)>\epsilon2^{5k/4}k^2 \right)\\
\leq \sum_{k=1}^\infty\frac{1}{2^{5k/2}k^4}E\left(\max_{n\leq 2^k}\sum_{1\leq i<j\leq n}g_2(\bX_i,\bX_j)\right)^2\leq C\sum_{k=1}^\infty  \frac{2^{5k/2}k^2}{2^{5k/2}k^4}=C\sum_{k=1}^\infty \frac{1}{k^2}<\infty,
\end{multline*}
and the almost sure convergence follows by the Borel--Cantelli lemma.
\end{proof}
%
%



Towards the proof of Theorem \ref{theo3}, we further state and prove Lemmas \ref{lem9} to \ref{lem11}.
\begin{lemma}\label{lem9} Under Assumptions \ref{ass3}, \ref{ass2} and \ref{ass1}, we have for any constants $(c_i)_{i\in\N}$
\begin{equation*}
E\bigg(\sum_{i=1}^ng_1(\bX_i)c_i\bigg)^2\leq Cn\left(\max_{i=1,\ldots,n}|c_i|\right)^2,
\end{equation*}
\end{lemma}
\begin{proof}[Proof of Lemma \ref{lem9}] By Assumption \ref{ass1} and Lemma \ref{lem:neu1}, the process $(g_1(\bX_n))_{n\in\Z}$ is $L_2$-NED with approximation constants $a_{l,2}=O(l^{-3})$. We have the following bound for the autocovariance: 
\[
	\left|Eg_1(\bX_i)g_1(\bX_{i+k})\right|
\]\[
	\leq	
	\left|E \{E(g_1(\bX_i)|\F_{i-l}^{i+l})E(g_1(\bX_{i+k})|\F_{i+k-l}^{i+k+l})\}\right|
	+ 
	\left|E\{g_1(\bX_i)\left(g_1(\bX_{i+k})-E(g_1(\bX_{i+k})|\F_{i+k-l}^{i+k+l})\right)\}\right|
\]\[
	\quad +\left|E\left(E(g_1(\bX_{i+k})|\F_{i+k-l}^{i+k+l})\left(g_1(\bX_{i})-E(g_1(\bX_{i})|\F_{i-l}^{i+l})\right)\right)\right|
\]\[
	\leq 10\left\|E(g_1(\bX_{i+k})|\F_{i+k-l}^{i+k+l})\right\|_{2+\delta}^2\beta^{\frac{\delta}{2+\delta}}_{k-2l}+2\left\|g_1(\bX_i)\right\|_2	
	\left\|g_1(\bX_{i+k})-E(g_1(\bX_{i+k})|\F_{i+k-l}^{i+k+l})\right\|_2
\]\[
	\leq C(a_{l,2}+\beta^{\frac{\delta}{2+\delta}}_{l}),
\]
where we used the inequality by \citet{davydov1970invariance} and set $l=\lfloor\frac{k}{3}\rfloor$. Recall that $\F_i^j$ is defined as the $\sigma$-algebra generated by $\bZ_i,\ldots,\bZ_j$. Now it follows from the stationarity of the process that
\begin{multline*}
E\left(\sum_{i=1}^ng_1(\bX_i)c_i\right)^2\leq \sum_{i,j=1}^n\left|Eg_1(\bX_i)g_1(\bX_{j})\right|\left|c_i\right|\left|c_j\right|
\\\leq \left(\max_{i=1,\ldots,n}|c_i|\right)^2\sum_{i,j=1}^n\left|Eg_1(\bX_i)g_1(\bX_{j})\right|\leq C\left(\max_{i=1,\ldots,n}|c_i|\right)^2n\sum_{k=0}^n\left|Eg_1(\bX_0)g_1(\bX_{k})\right|\\
\leq C\left(\max_{i=1,\ldots,n}|c_i|\right)^2n\sum_{l=0}^{\lfloor\frac{n}{3}\rfloor}(a_{l,2}+\beta^{\frac{\delta}{2+\delta}}_{l})\leq Cn\left(\max_{i=1,\ldots,n}|c_i|\right)^2,
\end{multline*}
and the Lemma is proved.
\end{proof}

\begin{lemma}\label{lem10} Under the Assumptions \ref{ass3}, \ref{ass2} and \ref{ass1} for any constants $(c_{ij})_{i,j \in\N}$
\begin{equation*}
\left|E\sum_{j_1,j_2,j_3,j_4=1}^ng_2(\bX_{j_1},\bX_{j_2})g_2(\bX_{j_3},\bX_{j_4})c_{j_1j_3}\right|\leq C\max_{i,j\in\{1,\ldots,n\}}\left|c_{ij}\right|n^{5/2}.
\end{equation*}
\end{lemma}

\begin{proof}[Proof of Lemma \ref{lem10}] Recall that we abbreviate $f_l(\bZ_{j-l},\ldots,\bZ_{j+l})$ by $\bX_{j,l}$. We use the triangle inequality to obtain
\[
\bigg|E\sum_{j_1,j_2,j_3,j_4=1}^ng_2(\bX_{j_1},\bX_{j_2})g_2(\bX_{j_3},\bX_{j_4})c_{j_1j_3}\bigg|\] \[
\leq \bigg|E\sum_{j_1,j_2,j_3,j_4=1}^ng_{2,l}({\bX}_{j_1,l},{\bX}_{j_2,l})g_{2,l}({\bX}_{j_3,l},{\bX}_{j_4,l})c_{j_1j_3}\bigg|\] \[
+\bigg|E\sum_{j_1,j_2,j_3,j_4=1}^n\left(g_{2,l}({\bX}_{j_1,l},{\bX}_{j_2,l})-g_{2}({\bX}_{j_1,l},{\bX}_{j_2,l})\right)\left(g_{2,l}({\bX}_{j_3,l},{\bX}_{j_4,l})-g_{2}({\bX}_{j_3,l},{\bX}_{j_4,l})\right)c_{j_1j_3}\bigg|\] \[
+\bigg|E\sum_{j_1,j_2,j_3,j_4=1}^n\left(g_{2}({\bX}_{j_1,l},{\bX}_{j_2,l})-g_{2}(\bX_{j_1},\bX_{j_2})\right)\left(g_{2}({\bX}_{j_3,l},{\bX}_{j_4,l})-g_{2}(\bX_{j_3},\bX_{j_4})\right)c_{j_1j_3}\bigg|\] \[
+\bigg|E\sum_{j_1,j_2,j_3,j_4=1}^ng_{2,l}({\bX}_{j_1,l},{\bX}_{j_2,l})\left(g_{2,l}({\bX}_{j_3,l},{\bX}_{j_4,l})-g_{2}({\bX}_{j_3,l},{\bX}_{j_4,l})\right)c_{j_1j_3}\bigg|\displaybreak[0]\] \[
+\bigg|E\sum_{j_1,j_2,j_3,j_4=1}^n\left(g_{2,l}({\bX}_{j_1,l},{\bX}_{j_2,l})-g_{2}({\bX}_{j_1,l},{\bX}_{j_2,l})\right)g_{2,l}({\bX}_{j_3,l},{\bX}_{j_4,l})c_{j_1j_3}\bigg|\] \[
+\bigg|E\sum_{j_1,j_2,j_3,j_4=1}^ng_{2,l}({\bX}_{j_1,l},{\bX}_{j_2,l}))\left(g_{2}({\bX}_{j_3,l},{\bX}_{j_4,l})-g_{2}(\bX_{j_3},\bX_{j_4})\right)c_{j_1j_3}\bigg|\displaybreak[0]\] \[
+\bigg|E\sum_{j_1,j_2,j_3,j_4=1}^n\left(g_{2}({\bX}_{j_1,l},{\bX}_{j_2,l})-g_{2}(\bX_{j_1},\bX_{j_2})\right)g_{2,l}({\bX}_{j_3,l},{\bX}_{j_4,l})c_{j_1j_3}\bigg|\] \[
+\bigg|E\sum_{j_1,j_2,j_3,j_4=1}^n\left(g_{2,l}({\bX}_{j_1l},{\bX}_{j_2,l})-g_{2}({\bX}_{j_1,l},{\bX}_{j_2,l})\right)\left(g_{2}({\bX}_{j_3,l},{\bX}_{j_4,l})-g_{2}(\bX_{j_3},\bX_{j_4})\right)c_{j_1j_3}\bigg|\] \[
+\bigg|E\sum_{j_1,j_2,j_3,j_4=1}^n\left(g_{2}({\bX}_{j_1,l},{\bX}_{j_2,l})-g_{2}(\bX_{j_1},\bX_{j_2})\right)\left(g_{2,l}({\bX}_{j_3,l},{\bX}_{j_4,l})-g_{2}({\bX}_{j_3,l},{\bX}_{j_4,l})\right)c_{j_1j_3}\bigg|\] \[
=I_n+I\!\! I_n+I\!\! I\!\! I_n+I\! V_n+V_n+V\!\! I_n+V\!\! I\!\! I_n+V\!\! I\!\! I\!\! I_n+I\!\! X_n.
\]
We will establish the bound only for some of the summands in order to keep this proof short. The bound for $I_n$ can be shown in the same way as Lemma \ref{lem5}. For $I\!\! I_n$, $I\!\! I\!\! I_n$, $V\!\! I\!\! I\!\! I_n$ and $I\!\! X_n$, we use the H\"older inequality and can then proceed similar to the proof of Lemma \ref{lem4} and \ref{lem4a}. For example by Lemmas \ref{lem3} and \ref{lem4a}, we have
\[
	V\!\! I\!\! I\!\! I_n
	\leq 
	\sum_{j_1,j_2,j_3,j_4=1}^n
	\left\|g_{2,l}({\bX}_{j_1,l},{\bX}_{j_2,l})-g_{2}({\bX}_{j_1,l},{\bX}_{j_2,l})\right\|_2
	\left\|g_{2}({\bX}_{j_3,l},{\bX}_{j_4,l})-g_{2}(\bX_{j_3},\bX_{j_4})\right\|_2c_{j_1j_3}
\]\[	
	\leq 
	C\max_{i,j\in\{1,\ldots,n\}}\left|c_{ij}\right|
	\sum_{j_1,j_2,j_3,j_4=1}^n
	(\sqrt{\epsilon}+a_l^{\frac{\delta}{2+\delta}}\Phi^{\frac{\delta}{2+\delta}}(\epsilon)+\beta^{\frac{\delta}{2+\delta}}_{|j_2-j_1|})
	(\sqrt{\epsilon}+a_l^{\frac{\delta}{2+\delta}}\Phi^{\frac{\delta}{2+\delta}}(\epsilon)+\beta^{\frac{\delta}{2+\delta}}_{|j_4-j_3|})
\]\[
	\leq C\max_{i,j\in\{1,\ldots,n\}}\left|c_{ij}\right|n^{5/2}
\] 
with $l=\lfloor n^{1/4}\rfloor$ and $\epsilon=l^{-6}$. For the summand $I\! V_n$, the H\"older inequality is used in a slightly different way
\[
	I\! V_n\leq \sum_{j_3,j_4=1}^n
	\bigg|
		E 
		\bigg\{ 
			\sum_{j_1,j_2=1}^ng_{2,l}({\bX}_{j_1,l},{\bX}_{j_2,l})
			\left(g_{2,l}({\bX}_{j_3,l},{\bX}_{j_4,l})-g_{2}({\bX}_{j_3,l},{\bX}_{j_4,l})\right)c_{j_1j_3} 
		\bigg\} 
	\bigg|
\]\[
	\leq \sum_{j_3,j_4=1}^n
	\bigg\|\sum_{j_1,j_2=1}^ng_{2,l}({\bX}_{j_1,l},{\bX}_{j_2,l})c_{j_1j_3}\bigg\|_2
	\big\|g_{2,l}({\bX}_{j_3,l},{\bX}_{j_4,l})-g_{2}({\bX}_{j_3,l},{\bX}_{j_4,l})\big\|_2
\]\[
	\leq 
	C\sum_{j_3,j_4=1}^n n^{5/4}\max_{i,j\in\{1,\ldots,n\}}\left|c_{ij}\right|(\sqrt{\epsilon}+a_l^{\frac{\delta}{2+\delta}}\Phi^{\frac{\delta}{2+\delta}}(\epsilon)+\beta^{\frac{\delta}{2+\delta}}_{|j_4-j_3|})\leq C\max_{i,j\in\{1,\ldots,n\}}\left|c_{ij}\right|n^{5/2}.
\]
A similar treatment of the summands $V_n$, $V\!\! I_n$ and $V\!\! I\!\! I_n$ completes the proof.
\end{proof}

\begin{lemma}\label{lem11} Under Assumptions \ref{ass3}, \ref{ass2}, \ref{ass1} and \ref{ass4}, we have, for $n \to \infty$,
\begin{equation*}
	E
	\bigg|
		\sum_{r=-n}^n\frac{1}{n}\sum_{i=1}^{n-|r|}\left(g_1(\bX_i)g_1(\bX_{i+|r|})-\hat{g}_1(\bX_i)\hat{g}_1(\bX_{i+|r|})\right)\kappa(|r|/b_n)
	\bigg|
	\rightarrow 0.
\end{equation*}
\end{lemma}

\begin{proof}[Proof of Lemma \ref{lem11}] We first expand the difference of $g_1$ and its estimator $\hat{g}_1$ as
\begin{equation*}
	g_1(\bx)-\hat{g}_1(\bx)=\frac{1}{n}\sum_{i=1}^ng_1(\bX_i)-\frac{1}{n}\sum_{i=1}^ng_2(\bx,\bX_i)+\frac{1}{n^2}\sum_{i,j=1}^ng_2(\bX_i,\bX_j).
\end{equation*}
With the help of this, we split the expectation into six parts and apply the triangle inequality:
\[ \textstyle
	E\Big|\sum_{r=-n}^n\frac{1}{n}\sum_{i=1}^{n-|r|}\left(g_1(\bX_i)g_1(\bX_{i+|r|})-\hat{g}_1(\bX_i)\hat{g}_1(\bX_{i+|r|})\right)\kappa(|r|/b_n)\Big|\]\[ \textstyle
	\leq E\Big|\sum_{r=-n}^n\frac{1}{n}\sum_{i=1}^{n-|r|}\left(g_1(\bX_i)-\hat{g}_1(\bX_i)\right)g_1(\bX_{i+|r|})\kappa(|r|/b_n)\Big|
\]\[ \textstyle
	+ E\Big|\sum_{r=-n}^n\frac{1}{n}\sum_{i=1}^{n-|r|}\left(g_1(\bX_{i+|r|})-\hat{g}_1(\bX_{i+|r|})\right)\hat{g}_1(\bX_{i})\kappa(|r|/b_n)\Big|
\]\[ \textstyle
	\leq E\Big|\sum_{r=-n}^n\frac{1}{n}\sum_{i=1}^{n-|r|}\frac{1}{n}\sum_{j=1}^ng_1(\bX_j)g_1(\bX_{i+|r|})\kappa(|r|/b_n)\Big|\displaybreak[0]
\]\[ \textstyle
	+ E\Big|\sum_{r=-n}^n\frac{1}{n}\sum_{i=1}^{n-|r|}\frac{1}{n}\sum_{j=1}^ng_2(\bX_i,\bX_j)g_1(\bX_{i+|r|})\kappa(|r|/b_n)\Big|\displaybreak[0]
\]\[ \textstyle
	+ E\Big|\sum_{r=-n}^n\frac{1}{n}\sum_{i=1}^{n-|r|}\frac{1}{n^2}\sum_{j_1,j_2=1}^ng_2(\bX_{j_1},\bX_{j_2})g_1(\bX_{i+|r|})\kappa(|r|/b_n)\Big|
\]\[ \textstyle
	+ E\Big|\sum_{r=-n}^n\frac{1}{n}\sum_{i=1}^{n-|r|}\frac{1}{n}\sum_{j=1}^ng_1(\bX_j)\hat{g}_1(\bX_{i})\kappa(|r|/b_n)\Big|
\]\[ \textstyle
	+ E\Big|\sum_{r=-n}^n\frac{1}{n}\sum_{i=1}^{n-|r|}\frac{1}{n}\sum_{j=1}^ng_2(\bX_{i+|r|},\bX_j)\hat{g}_1(\bX_{i})\kappa(|r|/b_n)\Big|
\]\[ \textstyle
	+ E\Big|\sum_{r=-n}^n\frac{1}{n}\sum_{i=1}^{n-|r|}\frac{1}{n^2}\sum_{j_1,j_2=1}^ng_2(\bX_{j_1},\bX_{j_2})\hat{g}_1(\bX_{i})\kappa(|r|/b_n)\Big|
\] \[
	\qquad = I_n+I\!\! I_n+I\!\! I\!\! I_n+I\! V_n+V_n+V\!\! I_n.
\]
For the first summand $I_n$, we use the H\"older inequality and Lemma \ref{lem9} with constants $c_i=\sum_{i_2=1}^n\kappa(|i-i_2|/b_n)=O(b_n)$ and obtain
\[
	\textstyle
	I_n = 
	E\Big|\frac{1}{n}\sum_{j=1}^ng_1(\bX_j)\Big|\cdot\Big|\frac{1}{n}\sum_{i=1}^{n}g_1(\bX_{i})\sum_{i_2=1}^n\kappa(|i-i_2|/b_n)\Big|
\]\[
	\qquad \leq 
	\textstyle
	\Big\|\frac{1}{n}\sum_{j=1}^ng_1(\bX_j)\Big\|_2
	\Big\|\frac{1}{n}\sum_{i=1}^{n}g_1(\bX_{i})\sum_{i_2=1}^n\kappa(|i-i_2|/b_n)\Big\|_2\leq C\frac{1}{\sqrt{n}}\frac{1}{\sqrt{n}}b_n\rightarrow 0
\] 
as $n\rightarrow\infty$ due to the assumptions on $b_n$. %
For $I\!\! I_n$, we use Lemma \ref{lem10} to obtain
\[ 
	\textstyle
	I\!\! I_n = E\bigg|\frac{1}{n}\sum_{i_1=1}^n\frac{1}{n}\sum_{i,j=1}^{n}g_2(\bX_i,\bX_j)g_1(\bX_{i_1})\kappa(|i-i_1|/b_n)\bigg|
\]\[
	\textstyle
	\leq 
	E
	\left[
		\big\{ \frac{1}{n}\sum_{i=1}^n g_1(\bX_{i_1})^2 \big\}^{1/2} 
		\Big\{
			\frac{1}{n}\sum_{i_1=1}^n \big(\frac{1}{n}\sum_{i,j=1}^{n}g_2(\bX_i,\bX_j)\kappa(|i-i_1|/b_n)\big)^2
		\Big\}^{1/2}
	\right]
\]\[
	\textstyle
	\leq \Big[ E \big\{  \frac{1}{n}\sum_{i=1}^n g_1(\bX_{i_1})^2  \big\} \Big]^{1/2}
\]\[
	\qquad \qquad
	\bigg[ E \big\{ 
		\frac{1}{n^3}\sum_{j_1,j_2,j_3,j_4=1}^ng_2(\bX_{j_1},\bX_{j_2})g_2(\bX_{j_3},\bX_{j_4})\sum_{i_1=1}^n\kappa(|j_1-i_1|/b_n)\kappa(|j_3-i_1|/b_n)
	\big\} \bigg]^{1/2}
\]\[
	\leq C\sqrt{\frac{n^{5/2}b_n}{n^3}}\rightarrow 0,
\]
since $c_{j_1j_2}=\sum_{i_1=1}^n\kappa(|j_1-i_1|/b_n)\kappa(|j_3-i_1|/b_n)=O(b_n)=o(\sqrt{n})$ and 
$	E\{\frac{1}{n}\sum_{i=1}^n g_1(\bX_{i_1})^2\}^2 \leq E\{ g_1(\bX_{0})^2 \}^2<\infty$
due to Assumption \ref{ass2}. For the third summand $I\!\! I\!\! I_n$, we use again the H\"older inequality and Lemma \ref{lem9} to get
\[	
	\textstyle
	I\!\! I\!\! I_n = 
	E\Big|\sum_{r=-n}^n\frac{1}{n}\sum_{i=1}^{n-|r|}\frac{1}{n^2}\sum_{j_1,j_2=1}^ng_2(\bX_{j_1},\bX_{j_2})g_1(\bX_{i+|r|})\kappa(|r|/b_n)\Big|
\]\[
	\leq 
	\bigg\| \frac{1}{n^2}\sum_{j_1,j_2=1}^ng_2(\bX_{j_1},\bX_{j_2})\bigg\|_2
	\bigg\|\frac{1}{n}\sum_{i=1}^{n-|r|}g_1(\bX_{i+|r|})\sum_{i_1=1}^n\kappa(|i-i_1|/b_n)\bigg\|_2
	\leq C\frac{1}{n^{3/4}}\frac{1}{\sqrt{n}}b_n
	\rightarrow 0.
\]
The convergence of the remaining parts $I\! V_n$, $V_n$ and $V\!\! I_n$ can be shown in the same way.
\end{proof}
%
%
%
%
%
%
%
%
%
%

\small